\documentclass[A4paper, 12pt]{article} 
\usepackage[T1]{fontenc}
\usepackage[utf8]{inputenc} 
\usepackage[all,cmtip]{xy}

\usepackage[margin=1in, top=1.2in, bottom=1.4in]{geometry}
\usepackage{graphicx} 

\usepackage{booktabs} 
\usepackage{array} 
\usepackage{paralist} 
\usepackage{verbatim} 
\usepackage{subfig} 
\usepackage{mathtools}
\usepackage{amssymb}
\usepackage{changepage}
\usepackage{amsmath}
\usepackage{amsthm}
\usepackage{textcomp}
\usepackage{mathdots}
\usepackage{mathrsfs}
\usepackage{enumitem}
\usepackage[colorlinks, allcolors=blue]{hyperref}
\usepackage{appendix}
\usepackage{stmaryrd}
\usepackage{hyperref}
\usepackage[bottom]{footmisc}
\hypersetup{colorlinks,linkcolor={blue},citecolor={blue},urlcolor={red}} 

\usepackage{subfiles}

\newcommand\blfootnote[1]{%
  \begingroup
  \renewcommand\thefootnote{}\footnote{#1}%
  \addtocounter{footnote}{-1}%
  \endgroup
}

\usepackage[tightpage,psfixbb]{preview}
\usepackage{sectsty}
\sectionfont{\large}
\subsectionfont{\normalfont \itshape}
\subsubsectionfont{\normalfont\itshape}
\makeatletter

\usepackage[nottoc,notlof,notlot]{tocbibind} 
\usepackage[titles,subfigure]{tocloft} 
\theoremstyle{sltheoremstyle}
\newtheorem{theorem}{Theorem}[section]
\newtheorem{lemma}[theorem]{Lemma}
\newtheorem{claim}[theorem]{Claim}
\newtheorem{corollary}[theorem]{Corollary}
\newtheorem{question}[theorem]{Question}
\newtheorem{proposition}[theorem]{Proposition}

\newtheorem*{corollary-non}{Corollary}
\newtheorem*{lemma-non}{Lemma}
\newtheorem*{theorem-non}{Theorem}
\newtheorem*{proposition-non}{Proposition}
\newtheorem*{condition-non}{Condition}
\newtheorem*{conditions-non}{Conditions}

\theoremstyle{definition}

\newtheorem{remark}[theorem]{Remark}

\newtheorem{definition}[theorem]{Definition}

\newtheorem{convention}[theorem]{Convention}

\usepackage[x11names]{xcolor}
\usepackage{tabularx}
\usepackage{geometry}
\usepackage{setspace}
\usepackage{lipsum}
\usepackage[explicit]{titlesec}
\usepackage{authblk}

\usepackage[backend=bibtex,style=alphabetic]{biblatex}

\newcommand*{\img}[1]{%
    \raisebox{-.02\baselineskip}{%
        \includegraphics[
        height=\baselineskip,
        width=\baselineskip,
        keepaspectratio,
        ]{#1}%
    }%
}

\newcommand{\wh}[1]{{\widehat{#1}}}

\newcommand{\set}[1]{\left\{ #1 \right\}}
\newcommand{\va}[1]{\left| #1 \right|}

\newcommand{\bracktorsion}{(\textnormal{torsion})}
\newcommand{\equivcohom}{\rm H^{2k}_G(X(\CC), \ZZ(k))}

\newcommand{\CC}{\mathbb{C}}

\newcommand{\OO}{\mathcal{O}}

\newcommand{\FF}{\mathbb{F}}

\newcommand{\QQ}{\mathbb{Q}}

\newcommand{\et}{\textnormal{\'et}}

\newcommand{\ch}{\textnormal{ch}}
\newcommand{\Hdg}{\textnormal{Hdg}}
\newcommand{\GL}{\textnormal{GL}}

\newcommand{\SL}{\textnormal{SL}}

\newcommand{\RR}{\mathbb{R}}

\newcommand{\ZZ}{\mathbb{Z}}

\newcommand{\CH}{\textnormal{CH}}

\newcommand{\Pic}{\textnormal{Pic}}

\newcommand{\Ker}{\textnormal{Ker}}
\newcommand{\NS}{\textnormal{NS}}

\newcommand{\tors}{\textnormal{tors}}

\newcommand{\id}{\textnormal{id}}

\newcommand{\alg}{\textnormal{alg}}

\newcommand{\Gal}{\textnormal{Gal}}

\newcommand{\Ima}{\textnormal{Im}}

\newcommand{\Aut}{\textnormal{Aut}}

\newcommand{\Spec}{\textnormal{Spec}}

\newcommand{\Sp}{\textnormal{Sp}}

\newcommand{\ca}[1]{{\mathcal{#1}}}
\newcommand{\bb}[1]{{\mathbb{#1}}}
\newcommand{\msf}[1]{{\mathsf{#1}}}
\newcommand{\mr}[1]{{\mathscr{#1}}}

\newcommand{\tn}[1]{{\textnormal{#1}}}
\let\rm\relax 
\newcommand{\rm}[1]{{\mathrm{#1}}}

\DeclareSymbolFont{bbm}{U}{bbm}{m}{n}
\DeclareSymbolFontAlphabet{\mathbbm}{bbm}

\begin{document}

\title{\Large{\textbf{On the integral Hodge conjecture\\for real abelian threefolds}}}
\author{Olivier de Gaay Fortman}
\date{\vspace{-7ex}}




\maketitle 


%

\begin{abstract}\noindent
We prove the real integral Hodge conjecture for several classes of real abelian threefolds. For instance, we prove the property for real abelian threefolds with connected real locus, and for real abelian threefolds $A$ which are the product $A = B \times E$ of an abelian surface $B$ and an elliptic curve $E$ with connected real locus $E(\RR)$. Moreover, we show that every real abelian threefold satisfies the real integral Hodge conjecture modulo torsion, and reduce the principally polarized case to the Jacobian case.
\end{abstract}


\section{Introduction}\label{sec:intro}

The Hodge structure on the singular cohomology ring of a smooth projective variety $X$ over $\CC$ is a powerful tool to study the algebraic cycles on $X$. Where the Hodge conjecture predicts that the space of rational Hodge classes in any degree $2k$ is spanned by algebraic cycles, the \emph{integral Hodge conjecture} asks for something stronger, namely that this is already true with $\ZZ$-coefficients. As such, the integral Hodge conjecture is a property rather than a conjecture: it may hold or fail depending on $k$ and the geometry of $X$ \cite{atiyahintegralhodge, trento, totarocobordism, voisinintegralhodge, colliotthelenevoisin, stabalyirrationalschreieder,benoistottem, totaroIHCthreefolds, beckmanndegaayfortman}. 

Recently, the analogue of the integral Hodge conjecture for real algebraic varieties has been formulated \cite{BW20, BW21}. Let $X$ be a smooth projective variety over $\RR$, and let $G = \Gal(\CC/\RR)$. Building on work of Krasnov \cite{krasnovcharact,krasnovgroth} and Van Hamel \cite{vanhamel}, Benoist and Wittenberg define a subgroup $\Hdg^{2k}_G(X(\CC), \ZZ(k))_0$ of the $G$-equivariant cohomology group $\rm H^{2k}_G(X(\CC), \ZZ(k))$ in the sense of Borel, and study the cycle class map 
\begin{align}\label{realcycleclassmap}
\CH_i(X) \to \rm \Hdg^{2k}_G(X(\CC), \ZZ(k))_0, \quad \quad i + k = \dim(X).
\end{align}
The \emph{real integral Hodge conjecture for $i$-cycles} refers to the property that  (\ref{realcycleclassmap}) is surjective. As in the complex situation, this property holds for every $X$ if $i \in \{\dim(X), \dim(X)-1, 0\}$ \cite{krasnovcharact, mangoltehamel, vanhamel, BW20}, but may fail for other values of $i \in \set{0, 1, \dotsc, \dim(X)}$. 

Complex uniruled threefolds, as well as threefolds $X$ over $\CC$ with $K_X = 0$, satisfy the integral Hodge conjecture by work of Grabowski, Voisin and Totaro \cite{grabowski, voisinintegralhodge, totaroIHCthreefolds}. In \cite[Question 2.16]{BW20}, Benoist and Wittenberg ask whether the same is true over $\RR$. In fact, in \cite{BW21} they provide positive answers for various classes of uniruled threefolds. For real Calabi-Yau varieties, however, nothing seems to be known. \blfootnote{\emph{Date:} \today}
In particular, one may ask:


\begin{question} \label{question}
Do real abelian threefolds satisfy the real integral Hodge conjecture? 
\end{question}
\noindent
The goal of this paper is to provide evidence towards a positive answer to Question \ref{question}. 
\\
\\
Let us explain our results. If $A$ is a real abelian variety, then for each $k \in \bb Z_{\geq 0}$ the following sequence is exact (see Lemma \ref{lemma:important0}):
\begin{align*}
0 \to \Hdg^{2k}_G(A(\CC), \ZZ(k))_0[2] \to \Hdg^{2k}_G(A(\CC), \ZZ(k))_0 \to \Hdg^{2k}(A(\CC), \ZZ(k))^G \to 0.
\end{align*}\begin{definition}A real abelian variety $A$ satisfies the \emph{real integral Hodge conjecture for $i$-cycles modulo torsion} if the following map is surjective:
\begin{align} \label{eq:modtors}
\CH_i(A) \to \Hdg^{2k}(A(\CC), \ZZ(k))^G,  \quad \quad i + k = \dim(A). 
\end{align}
We say that $A$ satisfies the \emph{real integral Hodge conjecture modulo torsion} if $A$ satisfies the real integral Hodge conjecture for $i$-cycles modulo torsion for 
every $i \in \set{0, 1, \dotsc, \dim(A)}$.  
\end{definition}
\noindent
Our first main result is as follows. 

\begin{theorem} \label{theorem1}
Every abelian threefold over $\RR$ satisfies the real integral Hodge conjecture modulo torsion. 
\end{theorem}
\noindent
By using the Hochschild-Serre spectral sequence (see Section \ref{subsec:hochschild}), one can calculate the torsion rank of the equivariant cohomology of a real abelian threefold. We obtain:

\begin{corollary} \label{IHCforconnected}
Let $A$ be an abelian threefold over $\RR$ such that $A(\RR)$ is connected. Then $A$ satisfies the real integral Hodge conjecture. 
\end{corollary}
\noindent
Our proof of Theorem \ref{theorem1} is inspired by Grabowski's proof of the integral Hodge conjecture for complex abelian threefolds \cite{grabowski}. It consists of two steps:
\begin{enumerate}
\item \label{firstreduction} Reduce the real integral Hodge conjecture for one-cycles modulo torsion for abelian varieties of dimension $g$ to the algebraicity of the minimal class $$\gamma_\theta =\theta^{g-1}/(g-1)! \in \Hdg^{2g-2}(A(\CC), \ZZ(g-1))^G$$ for every principally polarized abelian variety $(A, \theta)$ of dimension $g$ over $\RR$.   
\item \label{secondreduction} Reduce the algebraicity of $\gamma_\theta$ on a principally polarized real abelian threefold $A$ to the case where $A = J(C)$ is the Jacobian of a real algebraic curve $C$ of genus three with non-empty real locus, where this is clear. 
\end{enumerate}
\noindent
An essential ingredient for the first reduction step is the fact that any polarized abelian variety over $\RR$ is isogenous to a principally polarized one. Although the analogue of this statement over any algebraically closed field is classical \cite{MumfordAV}, it fails over general fields \cite{Howe2001}. We will prove this fact in Theorem \ref{principalisogeny}. 

As for step \ref{secondreduction}, our strategy is to use Hecke orbits.
For $g \geq 1$, let $\ca A_g$ be the moduli stack of principally polarized abelian varieties of dimension $g$, and consider the real-analytic topology on its real locus $\ca A_g(\RR)$ (c.f.\ \cite[Definition 7.5 \& Theorem 8.1]{MR4479836} or \cite[Definition 2.7 \& Theorem 2.15]{thesis-degaayfortman}). For integers $a$ and $b$, define the \emph{$(a, b)$-Hecke orbit} of a moduli point $[(A, \lambda)] \in \ca A_g(\RR)$ as the set of $[(B, \mu)] \in \ca A_g(\RR)$ admitting an isogeny $A \to B$ preserving the polarizations up to a product of powers of $a$ and $b$ (see Definition \ref{def:heckeorbits}). Hecke orbits are well-known to be dense in $\ca A_g(\CC)$ (see 
\cite[Lemma 7.14]{thesis-degaayfortman}); we obtain the following real analogue. 
\begin{theorem} \label{density}
Let $p$ and $q$ be distinct odd prime numbers. 
The $(p,q)$-Hecke orbit of any point $x \in \ca A_g(\RR)$ is analytically dense in the connected component of $\ca A_g(\RR)$ containing $x$.  
\end{theorem}
\begin{corollary} \label{usefulcorollary}
Let $p$ and $q$ be as above. Every principally polarized abelian threefold over $\RR$ is isogenous, via an isogeny that preserves the polarizations up to a product of powers of $p$ and $q$, to the Jacobian of a non-hyperelliptic curve with non-empty real locus. 
\end{corollary}

\noindent
Reduction step \ref{secondreduction} follows because if an odd multiple of $\theta^2/2$ is algebraic for a principally polarized real abelian threefold $(A, \theta)$, then $\theta^2/2$ is algebraic. In the course of proving Theorem \ref{theorem1}, we also show that the main theorem of \cite{beckmanndegaayfortman} has the following analogue over $\RR$: a principally polarized real abelian variety $(A, \theta)$ satisfies the real integral Hodge conjecture for one-cycles modulo torsion if and only if $\gamma_\theta$ is algebraic (see Theorem \ref{grabowski}). 

Corollary \ref{usefulcorollary} turns out to be useful for the general principally polarized case as well:

\begin{theorem} \label{reduction}
Let $\ca A_3(\RR)^+$ be a connected component of the moduli space of principally polarized real abelian threefolds. Assume the real integral Hodge conjecture holds for every Jacobian $J(C)$ such that $[J(C)] \in \ca A_3(\RR)^+$ and the real locus $C(\RR)$ of $C$ is non-empty. Then the real integral Hodge conjecture holds for every real abelian variety in $\ca A_3(\RR)^+$. 
\end{theorem}
\noindent
In view of Theorems \ref{theorem1} and \ref{reduction}, Question \ref{question} restricted to the principally polarized case can be rephrased as follows. Let $C$ be a real algebraic curve of genus three with non-empty real locus. Is the torsion subgroup of $\Hdg^{4}_G(J(C)(\CC),\ZZ(2))_0$ algebraic? 

Our next result reduces this question further. The theorem concerns torsion cohomology classes of degree four on real abelian varieties $A$ of any dimension $g$. 
Let $F^\bullet$ be the filtration on $\Hdg^{2k}_G(A(\CC), \ZZ(k))_0 \subset \rm H^{2k}_G(A(\CC),\ZZ(k))$ induced by the Hochschild-Serre spectral sequence (see Section \ref{subsec:hochschild}). 

\begin{theorem} \label{fourierreduction}
Let $A$ be an abelian variety over $\RR$. The group $F^3\Hdg^{4}_G(A(\CC), \ZZ(2))_0$ is zero and the group $F^2\Hdg^{4}_G(A(\CC), \ZZ(2))_0$ is algebraic. 
\end{theorem}
\noindent
The proof of Theorem \ref{fourierreduction} relies on Proposition 3.11 in \cite{beckmanndegaayfortman} and an analysis of the Abel-Jacobi map for zero-cycles. 

For an abelian variety $A$ over $\RR$, the Hochschild-Serre spectral sequence degenerates \cite{krasnovharnackthom}. 
Define $\rm H^1(G, \rm H^{3}(A(\CC), \ZZ(2)))_0$ as the image of the canonical homomorphism 
\begin{align} \label{align:F-one}
 F^1\Hdg^{4}_G(A(\CC), \ZZ(2))_0 \to \rm H^1(G, \rm H^{3}(A(\CC), \ZZ(2))).
\end{align}
In fact, for abelian threefolds $A$ over $\RR$, the homomorphism (\ref{align:F-one}) is surjective unless $\va{\pi_0(A(\RR))} = 8$ (Corollary \ref{corollary:n(A)}). Combining Theorems \ref{theorem1} and \ref{fourierreduction} with the compatibility of the cycle class map (\ref{realcycleclassmap}) and the real Abel-Jacobi map (Theorem \ref{abeljacobicommutativity}), we obtain:

\begin{corollary} \label{questionreductioncorollary}
Let $A$ be an abelian threefold over $\RR$. Then $A$ satisfies the real integral Hodge conjecture 
if and only if the Abel-Jacobi map 
\[
\CH_1(A)_{\textnormal{hom}} \to \rm H^1(G, \rm H^3(A(\CC), \ZZ(2)))_0
\]
is surjective, where $\CH_1(A)_{\textnormal{hom}}$ denotes the kernel of $\CH_1(A) \to \rm H^4(A(\CC), \ZZ(2))$. 
\end{corollary}
\noindent
See Section \ref{section:AbelJacobi} for the definition of the real Abel-Jacobi map. 
One can then use Corollary \ref{questionreductioncorollary} and the K\"unneth formula 
to prove the following unconditional result. 

\begin{theorem} \label{th:abeliansurface}
Let $A = B \times E$ be the product of a real abelian surface $B$ and a real elliptic curve $E$ with $E(\RR)$ connected. Then $A$ satisfies the real integral Hodge conjecture. 
\end{theorem}
\noindent
It follows that there are no topological obstructions to the real integral Hodge conjecture for those real abelian threefolds $A$ with $\va{\pi_0(A(\RR))}  \neq 8$. In the case $A = E^3$ for a real elliptic curve $E$ with disconnected $E(\RR)$, we reduce the real integral Hodge conjecture to the algebraicity of one specific class $\omega \in \rm H^1(G, \rm H^3(A(\CC), \ZZ(2)))_0$, see Proposition \ref{prop:finalprop}. 

Finally, we establish a connection between the real integral Hodge conjecture and the Griffiths group of an abelian threefold $A$ over $\RR$. Let $\mathscr C^2(A) \subset \rm{Griff}^2(A_\CC)^G \otimes \ZZ/2$ be the subgroup generated by pull-backs of Ceresa cycles $[\Xi] \in \rm{Griff}^2(J(C)_\CC)^G \otimes \ZZ/2$ along isogenies $\phi \colon A \to J(C)$ to Jacobians of curves of genus three over $\RR$ (c.f.\ Section \ref{sec:ceresa}). 
\begin{theorem} \label{thm:ceresa}Let $A$ be a real abelian threefold such that $A(\RR)$ is not connected. Suppose that the complex abelian threefold $A_\CC$ is very general. 
If $A$ satisfies the real integral Hodge conjecture, then $\mr C^2(A) \subsetneq \rm{Griff}^2(A_\CC)^G \otimes \ZZ/2$. 
\end{theorem}

\subsection{Acknowledgements}

This paper is based on results that appear in the last chapter of my PhD thesis \cite{thesis-degaayfortman}. I thank my advisor Olivier Benoist for the great guidance he has given me over the past three years. In particular, I owe him much for drawing my attention to the real integral Hodge conjecture, and for the many discussions we had concerning this project. 

I thank Nicolas Tholozan for several discussions on moduli of real abelian varieties, and in particular for explaining to me how to prove Theorem \ref{density}. Moreover, I thank Fabrizio Catanese, Fran\c{c}ois Charles and Javier Fres\'an for stimulating conversations, and I thank Olivier Benoist and Olivier Wittenberg for useful comments on an earlier version of this paper. 

This research received funding from the European Union’s Horizon 2020 research and innovation programme under the Marie Skłodowska-Curie grant agreement N\textsuperscript{\underline{o}}754362 \img{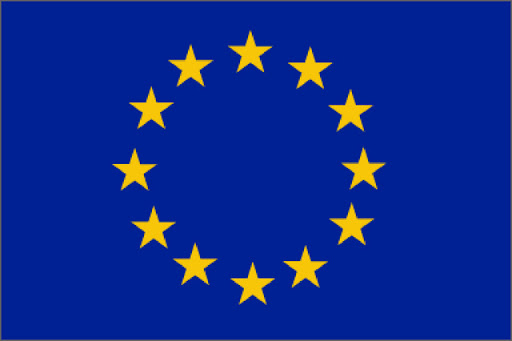}. 

\section{The real integral Hodge conjecture} \label{sec:realintegralhodge}

\subsection{Generalities}Let $X$ be a smooth projective variety over $\RR$. The group \[G = \Gal(\CC/\RR) = \{ \id, \sigma \}\] acts on $X(\CC)$ via the canonical anti-holomorphic involution $\sigma \colon X(\CC) \to X(\CC)$. For $k \in \ZZ$, we denote by $\ZZ(k)$ the $G$-module that has $\ZZ$ as underlying $\ZZ$-module, on which $G$ acts by $\sigma(1) = (-1)^k$. Thus, $\ZZ(k) = \ZZ(q)$ for every $q \in \ZZ$ with $k \equiv q \bmod 2$. We also denote by $\ZZ(k)$ the constant $G$-sheaf on $X(\CC)$ attached to the $G$-module $\ZZ(k)$. For $k,q \in \ZZ_{\geq 0}$, the $G$-action on the group $\rm H^{k}(X(\CC), \ZZ(q))$ is understood to be the one induced by the involution $\rm H^k(\sigma) \circ F_\infty$, where $F_\infty = \sigma^\ast$ is the pull-back of the anti-holomorphic involution $\sigma$ on $X(\CC)$, and $\rm H^k(\sigma)$ is the involution on cohomology induced by $\sigma \colon \ZZ(q) \to \ZZ(q)$. 

Let $k \in \ZZ_{\geq 0}$. Attached to $X$ is also the so-called degree $2k$ \emph{equivariant cohomology group} with coefficients in $\ZZ(k)$, see \cite{tohoku}. It is denoted by $\rm H_G^{2k}(X(\CC), \ZZ(k))$, and relates to singular cohomology via a canonical homomorphism
\begin{align}\label{equivariant-to-singular}
\varphi \colon \rm H^{2k}_G(X(\CC), \ZZ(k)) \to \rm H^{2k}(X(\CC), \ZZ(k))^G.
\end{align}
A real subvariety $Z \subset X$ of codimension $k$ induces a class $[Z] \in \rm H^{2k}_G(X(\CC), \ZZ(k))$ whose image $\varphi([Z])$ in $\rm H^{2k}(X(\CC), \ZZ(k))^G$ is the class $[Z_\CC]$ attached to the subvariety $Z_\CC \subset X_\CC$. 

It turns out that such algebraic cycle classes satisfy an additional condition, discovered by Kahn and Krasnov \cite{kahn, krasnovgroth}, and defined purely in terms of the structure of $X(\CC)$ as topological $G$-space. For $i \in \{0, \dotsc, 2k\}$, define 
\begin{align*}
\phi_i \colon \rm H^{2k}_G(X(\CC),\ZZ(k)) \to \rm H^i(X(\RR),\ZZ/2)
\end{align*} 
as the composition 
\begin{align*}
\rm    H^{2k}_G(X(\CC), \ZZ(k)) &\xrightarrow{\bmod 2}  \rm H^{2k}_G(X(\CC), \ZZ / 2) \\
&\xrightarrow{\text{restriction}} \rm H^{2k}_G(X(\RR), \ZZ / 2)  =  \rm H^{2k}(X(\RR) \times BG , \ZZ/2) \\
  & \xrightarrow[\sim]{\text{K\"unneth}} \rm H^0(X(\RR), \ZZ/2) \oplus \cdots \oplus \rm H^{2k}(X(\RR), \ZZ/2) \\
&  \xrightarrow{\text{projection}} \rm H^i(X(\RR),\ZZ/2). 
\end{align*}
For $\alpha \in \equivcohom$, define $\alpha_i = \phi_i(\alpha) \in \rm H^i(X(\RR),\ZZ/2)$, and let $\rm H^{2k}_G(X(\CC),\ZZ(k))_0$ be the group of classes $\alpha \in \rm H^{2k}_G(X(\CC), \ZZ(k))$ such that
\begin{align}    \label{topologicalcondition}
(\alpha_{0}, \alpha_{1}, \dotsc, \alpha_{k}, \dotsc, \alpha_{2k}) = 
    \left(0, \dotsc, 0, \alpha_{k}, Sq^1(\alpha_{k}), Sq^2(\alpha_{k}), \dotsc, Sq^k(\alpha_{k})\right).
 \end{align}
 Here, the $Sq^i$ are the Steenrod operations
$$
Sq^i: \rm H^p(X(\RR), \ZZ/2) \to \rm H^{p+i}(X(\RR), \ZZ/2).
$$
\begin{definition}[Benoist--Wittenberg] \label{def:BW}
The subgroup
\begin{align*}
\Hdg^{2k}_G(X(\CC), \ZZ(k))_0 \subset \rm H^{2k}_G(X(\CC), \ZZ(k))
\end{align*}
is the group of classes $\alpha \in \rm H^{2k}_G(X(\CC), \ZZ(k))$ that satisfy the following conditions:
\begin{enumerate}
    \item \label{def:top} The class $\alpha$ lies in the subgroup $\rm H^{2k}_G(X(\CC), \ZZ(k))_0$ of classes satisfying (\ref{topologicalcondition}). 
    \item \label{def:hodge} The image of $\alpha$ in $\rm H^{2k}(X(\CC), \ZZ(k))$ under (\ref{equivariant-to-singular}) is a Hodge class. 
\end{enumerate}
\end{definition}
\noindent
By \cite[\S1.6.4]{BW20}, Definition \ref{def:BW} is compatible with cup products, pull-backs and proper push-forwards. Moreover, by \cite{krasnovgroth} (c.f.\ \cite[Theorem 1.18]{BW20}), we have that $[Z] \in \Hdg^{2k}_G(X(\CC), \ZZ(k))_0$ for every algebraic subvariety $Z \subset X$. This defines the map (\ref{realcycleclassmap}). 

\subsection{Hochschild-Serre} \label{subsec:hochschild}  For a smooth projective variety $X$ over $\RR$, the \emph{Hochschild-Serre} spectral sequence
\begin{align}\label{hochschild}
E^{p,q}_2 = \rm H^p(G, \rm H^q(X(\CC), \ZZ(k))) \implies \rm H^{p+q}_G(X(\CC), \ZZ(k))
\end{align}
is obtained by viewing $\rm H^{i}_G(X(\CC), -)$ as the right-derived functor of the composition of taking global sections and $G$-invariants on the category of $G$-sheaves on $X(\CC)$. 

Let $A$ be an abelian variety over $\RR$. Then (\ref{hochschild}) degenerates for $k = 0$ by \cite[\S5.7]{krasnovharnackthom}. It follows that (\ref{hochschild}) degenerates for every $k$, as we see using cup-products with the non-trivial elements in $\rm H^1(G, \ZZ(1))$ and $\rm H^2(G, \ZZ)$. 
Consequently, there are canonical identifications
\begin{align} \label{canonical-HS}
\begin{split}
\rm H^{2k}_G(A(\CC), \ZZ(k))_{\tors} &= \rm H^{2k}_G(A(\CC), \ZZ(k))[2] \\
&=  \Ker\left( \rm H^{2k}_G(A(\CC), \ZZ(k)) \to \rm H^{2k}(A(\CC), \ZZ(k))\right) \\
&=  F^1\rm H^{2k}_G(A(\CC), \ZZ(k)). 
\end{split}
\end{align}
Moreover, these $\ZZ/2$-modules are (non-canonically) isomorphic to the $\ZZ/2$-module $$\bigoplus_{\substack{p + q = k \\ p> 0}} \rm H^p(G, \rm H^q(A(\CC), \ZZ(k))).$$

\subsection{The topological condition} \label{subsec:topological}
We have the following fundamental:
\begin{lemma} \label{lemma:important0}
Let $X$ be a smooth projective variety of dimension $n$ over $\RR$, and let $k = n-1$. Suppose that $X(\CC)$ has torsion-free degree $2k$ integral singular cohomology and that the Hochschild-Serre spectral sequence (\ref{hochschild}) degenerates. 
Then each row and each column in the following commutative diagram is exact:
\begin{align} \label{important0}
\begin{split}
\xymatrixrowsep{1.5pc}
\xymatrixcolsep{1.5pc}
\xymatrix{
0 \ar[r]& \rm H^{2k}_G(X(\CC), \ZZ(k))_0[2]  \ar@{^{(}->}[d] \ar[r]& \rm H^{2k}_G(X(\CC), \ZZ(k))_0\ar@{^{(}->}[d] \ar[r]^{\varphi \hspace{1mm}} & \rm H^{2k}(X(\CC), \ZZ(k))^G \ar[r] \ar@{=}[d]& 0 \\
0 \ar[r]& \rm H^{2k}_G(X(\CC), \ZZ(k))[2] \ar@{->>}[d] \ar[r]& \rm H^{2k}_G(X(\CC), \ZZ(k)) \ar[r]^{\varphi\hspace{1mm}} \ar@{->>}[d]& \rm H^{2k}(X(\CC), \ZZ(k))^G \ar[r]& 0 \\
& \bigoplus_{p \geq 1} \rm H^{k-2p}(X(\RR),\ZZ/2) \ar@{=}[r]& \bigoplus_{p \geq 1} \rm H^{k-2p}(X(\RR),\ZZ/2) . &&
}
\end{split}
\end{align}
\end{lemma}

\begin{proof}
By the degeneration of the Hochschild-Serre spectral sequence (\ref{hochschild}), the map $$\varphi \colon \rm H^{2k}_G(X(\CC), \ZZ(k)) \to \rm H^{2k}(X(\CC), \ZZ(k))^G$$ is a surjective homomorphism between abelian groups of the same rank. The target of $\varphi$ is torsion-free, so its kernel is $ \rm H^{2k}(X(\CC), \ZZ(k))[2]$, which explains the horizontal exact sequence in the middle of diagram (\ref{important0}). It follows from the proof of \cite[Proposition 1.8]{BW20} together with \cite[Equation (1.33) \& Remark 1.20.(i)]{BW20} that the middle vertical sequence in diagram (\ref{important0}) is split exact. This implies that, in diagram (\ref{important0}), the vertical arrow on the bottom left and the horizontal map $\varphi$ on the top right are both surjective. 
\end{proof}

\noindent
Note that the first and second horizontal sequence in diagram (\ref{important0}) remain exact after restricting to Hodge classes. 


\subsection{Threefolds}\label{subsec:threefolds} Let $X$ be a smooth projective threefold over $\RR$. The topological condition (\ref{topologicalcondition}) on degree four classes takes a particularly simple form: a class $\alpha \in \rm H^4_G(X(\CC),\ZZ(2))$ lies in $\rm H^4_G(X(\CC),\ZZ(2))_0$ if and only if 
\[
\alpha|_x = 0 \in \rm H^4_G(\{x\}, \ZZ(2)) = \ZZ/2 \quad \textnormal{for any} \quad x  \in X(\RR).  
\]
Moreover, the conditions $\alpha|_x = 0$ for $x$ in different connected components of $X(\RR)$ are linearly independent over $\ZZ/2$: for $n = 3$, the middle vertical sequence in (\ref{important0}) is the split exact sequence
\begin{align*} 
0 \to \rm H^4_G(X(\CC),\ZZ(2))_0 \to \rm H^4_G(X(\CC),\ZZ(2)) \xrightarrow{\phi_0} \rm H^0(X(\RR),\ZZ/2) \to 0. 
\end{align*}
Finally, since $X$ satisfies the real integral Hodge conjecture for $i$-cycles whenever $i \in \{0,2,3\}$ (see \cite[\S 2.3.1 and \S 2.3.2]{BW20}), the real integral Hodge conjecture for $X$ is equivalent to the surjectivity of the homomorphism $$\CH_1(X) \to \Hdg^4_G(X(\CC),\ZZ(2))_0.$$

\section{Density of Hecke orbits}


\subsection{Polarized real abelian varieties} \label{sec:polarizedabelian}
Let $A$ be a real abelian variety, by which we mean an abelian variety over $\RR$. Here, and in the sequel, the dual abelian variety of $A$ is denoted by $\wh A$. Define $\Lambda = \rm H_1(A(\CC), \ZZ)$. Denote by $\sigma \colon A(\CC) \to A(\CC)$ the canonical anti-holomorphic involution, and by $F_{\infty} \colon \Lambda \to \Lambda$ its push-forward. There is a canonical bijection between: 

\begin{itemize}
\item Symmetric isogenies $\lambda \colon A \to \wh A$ such that $\lambda_\CC = \varphi_{\ca L}$ is the homomorphism $\varphi_{\ca L} \colon A_\CC \to \wh{A}_\CC$ induced by an ample line bundle $\ca L$ on $A_\CC$ as in \cite{MumfordAV}. 
\item Alternating forms $E \colon \Lambda \times \Lambda \to \ZZ$ such that $F_\infty^\ast(E) = - E$ and such that the following hermitian form is positive definite: 
\[
H \colon \Lambda_\RR \times \Lambda_\RR \to \CC, \quad H(x,y) = E(ix,y) + iE(x,y).
\]
\item Classes of ample line bundles $\theta \in \text{NS}(A_\CC)^G = \Hdg^2(A(\CC),\ZZ(1))^G$. 
\end{itemize}
\noindent
In the sequel, a \emph{polarization} on $A$ will be an element in either one of the three sets above; the context will make clear which structure is meant. 

\subsection{Moduli of real abelian varieties}  \label{sec:moduliofabelianRRR}
Let $(A, \lambda)$ be a principally polarized complex abelian variety of dimension $g$. By Galois descent \cite{weilgalois}, to give a model of $(A, \lambda)$ over $\RR$ is to give an anti-holomorphic involution 
\[\sigma \colon A(\CC) \to A(\CC) \quad \textnormal{ such that } \quad \sigma(0) = 0 \quad \textnormal{ and } \quad F_\infty^\ast(E) = -E.
\]
By \cite[Chapter IV, Theorem (4.1)]{silholsurfaces} (or \cite[Section 9]{grossharris}), such an anti-holomorphic involution $\sigma$ exists if and only if the complex principally polarized abelian variety $(A, \lambda)$ admits a period matrix of the form 
\begin{align}\label{chosenperiodmatrix}
(I_g, \frac{1}{2}M + \rm{i} N).
\end{align}
Here $N$ is a positive definite real matrix and $M$ is a symmetric $g \times g$-matrix with integral coefficients such that if $r = \textnormal{rank}(M) \leq g$, then $M$ is of the form
\begin{align}\tag{1} \label{typeone}
\begin{pmatrix} 
I_r & 0 \\ 
0 & 0 
    \end{pmatrix}
    \end{align}
    or of the form
    \begin{align} \tag{2} \label{typetwo}
    \begin{pmatrix} 
    0 & \dots & 1 & 0 \\
    \vdots&\iddots&\iddots &0\\
    1&0  & \iddots &\vdots\\
    0 &  0 & \dots &0 
    \end{pmatrix}.
\end{align}

\begin{definition}[Silhol] \label{typedefinition}
The \emph{type} $(r, \alpha) \in \ZZ^2$ of a principally polarized real abelian variety $(A, \lambda)$ is defined as follows. 
If $(I_g, \frac{1}{2}M + \rm{i} N)$ is a period matrix for $(A_\CC, \lambda_\CC)$ as above, then $r = \text{rank}(M)$. Define $\alpha \in \{0, 1,2\}$ in the following way:
\begin{itemize}
\item If $r$ is odd, then $\alpha = 1$. Thus the type of $(A, \lambda)$ is $(r, 1)$. 
\item If $r$ is zero, then $\alpha = 0$. Thus the type of $(A, \lambda)$ is $(0,1)$. 
\item If $r$ is even, but non-zero, then $\alpha = 1$ if $M$ is of the form (\ref{typeone}) and $\alpha = 2$ if $M$ is of the form (\ref{typetwo}). 
\end{itemize} 
\end{definition}
\noindent
This definition makes sense, because of the following:

\begin{proposition}[Silhol]
The type $(r, \alpha)$ of a principally polarized real abelian variety $(A, \lambda)$ does not depend on the chosen period matrix (\ref{chosenperiodmatrix}) for $(A_\CC, \lambda_\CC)$ nor on the isomorphism class of $(A, \lambda)$. If $(A, \lambda)$ is of type $(r, \alpha)$, there exists a period matrix $(I_g, \frac{1}{2}M + \rm{i} N)$ for $(A, \lambda)$ such that $M$ is of the form (\ref{typeone}) or (\ref{typetwo}), according to whether $\alpha$ equals $1$ or $2$.\end{proposition}
\begin{proof}
See \cite[Chapter IV, Corollaries (4.3) and (4.5)]{silholsurfaces}.  
\end{proof}


\begin{definition}
Let $\mr T(g)$ be the set of types $(r, \alpha)$ of principally polarized abelian varieties of dimension $g$ over $\RR$. For any type $\tau \in \mr T(g)$, define $M(\tau)$ to be the integral $g \times g$-matrix (\ref{typeone}) or (\ref{typetwo}) above, according to whether $\alpha$ equals $1$ or $2$. Then define $\GL_g^\tau(\ZZ)$ to be the subgroup of $\GL_g(\ZZ)$ of matrices $T \in \GL_g(\ZZ)$ that satisfy 
\[
T^t\cdot M(\tau) \cdot T \equiv M(\tau) \mod 2 \quad \quad (T^t = \textnormal{ transpose of } T). 
\]
Finally, let $\rm H_g$ be the set of symmetric positive definite real matrices of rank $g$. 
\end{definition}

\begin{theorem}[Silhol]
Let $g$ be a positive integer. For $\tau \in \mr T(g)$, define $\va{ \ca A_g(\RR)}^\tau$ to be the set of isomorphism classes of real principally polarized abelian varieties of type $\tau$. For each type $\tau \in \mr T(g)$, the period map induces a bijection
\[
\va{\ca A_g(\RR)}^\tau  = \GL_g^\tau(\ZZ) \setminus \rm H_g,
\]
where $\GL_g^\tau(\ZZ)$ acts on $\rm H_g$ by $N \mapsto {T^t}\cdot N \cdot T$. Therefore, 
\begin{align}\label{silholsbijection}
\va{\ca A_g(\RR)} = \bigsqcup_{\tau \in \mr T(g)}  \GL_g^\tau(\ZZ) \setminus \rm H_g. 
\end{align}
\end{theorem}
\begin{proof}
See \cite[Chapter IV, Theorem (4.6)]{silholsurfaces}.
\end{proof}

\begin{remark} \label{remark:bijection=homeo}
By \cite[Theorem 8.1]{MR4479836}, 
the bijection (\ref{silholsbijection}) is a homeomorphism with respect to the real-analytic topology on $\va{\ca A_g(\RR)}$ (see \cite[Definition 7.5]{MR4479836}). 
\end{remark}


\subsection{Density of Hecke orbits over the real numbers} Before we prove Theorem \ref{density}, let us properly introduce the notion of Hecke orbits over the real numbers. 
\begin{definition} \label{def:heckeorbits}
Let $(A, \lambda)$ be a principally polarized abelian variety of dimension $g$ over $\RR$, and let $x = [(A, \lambda)] \in \va{\ca A_g(\RR)}$ the corresponding moduli point. For a tuple of integers $(a, b)$, the \emph{$(a,b)$-Hecke orbit of $x$} is the subset $\ca G_{a,b}(x) \subset \va{\ca A_g(\RR)}$ of isomorphism classes $[(B, \nu)] \in \va{\ca A_g(\RR)}$ of principally polarized abelian varieties $(B, \nu)$ of dimension $g$ over $\RR$, for which there exist $n,m \in \ZZ_{\geq 0}$ and an isogeny 
\[
\phi \colon A \to B \quad \textnormal{ such that } \quad \phi^\ast(\nu) = a^nb^m \cdot \lambda. 
\]
\end{definition}

 \begin{proof}[Proof of Theorem \ref{density}] 
 Let $p$ and $q$ be distinct odd prime numbers. 
\begin{enumerate}[leftmargin=-0.08cm, rightmargin = -0.08cm]
    \item[Step 1:] \emph{If $x = [(A, \lambda)] \in \va{\ca A_g(\RR)}$ and $\tau \in \mr T(g)$ is the type of $(A, \lambda)$, then $\ca G_{p,q}(x) \subset \va{\ca A_g(\RR)}^\tau$}. Indeed, for any $y = [(B, \mu)] \in \ca G_{p,q}(x)$, there exists an isogeny $\phi \colon A \to B$ such that $\phi^\ast(\mu) = n \cdot \lambda$ for some odd positive integer $n$. Such a map $\phi$ induces an isomorphism $$\rm H_1(A(\CC),\ZZ/2) \cong \rm H_1(B(\CC),\ZZ/2)$$ of symplectic spaces with involution. Since $x \in \va{\ca A_g(\RR)}^\tau$, this implies that $y \in \va{\ca A_g(\RR)}^\tau$ as well, see \cite[Section 9]{grossharris}. 
    \item[Step 2:] Define $S = \ZZ
\left[
\frac{1}{p}, \frac{1}{q}
\right]$. The ring homomorphism $S \to S / 2S$ induces a group homomorphism 
$
\GL_g(S) \to \GL_g(S/2S). 
$ For $\tau \in \mr T(g)$, we define
\[
\GL_g^\tau(S)= \{T \in \GL_g(S) \colon T^t \cdot M(\tau) \cdot T \equiv M(\tau) \bmod 2\}.
\]
 Fix one such $\tau \in \mr T(g)$. Observe that the action of $\GL_g^\tau(\ZZ)$ on $\rm H_g$ extends to a transitive action of $\GL_g(\RR)$ on $\rm H_g$. We claim:
 
     \emph{Let $x = [(A_x , \lambda_x)] \in \va{\ca A_g(\RR)}^\tau$, and lift $x$ to a point $y \in \rm H_g$. Consider the orbit $\GL_g^\tau(S) \cdot y \subset \rm H_g$ as well as its image $\GL_g^\tau(\ZZ)\setminus   \left(    \GL_g^\tau(S) \cdot y  \right) $ in $\va{\ca A_g(\RR)}^\tau = \GL_g^\tau(\ZZ) \setminus \rm H_g$. Then}
    \[
\GL_g^\tau(\ZZ) \setminus  \left(    \GL_g^\tau(S) \cdot y  \right) = \ca G_{p,q}(x). 
    \]
    Indeed, if $\bb H_g$ is the genus $g$ Siegel space of symmetric, complex $g\times g$ matrices $Z = X + \rm{i} Y$ whose imaginary part $Y$ is positive definite, then the inclusion 
    \[
    \rho_\tau \colon \rm H_g \hookrightarrow \bb H_g, \quad N \mapsto \frac{1}{2} \cdot M(\tau) + \rm{i}N\]
     is equivariant for the embedding
    \[
f_\tau\colon    \GL_g(\RR) \hookrightarrow \Sp_{2g}(\RR), \quad T \mapsto \begin{pmatrix} T^t & \frac{1}{2}\left(M(\tau) \cdot T^{-1}-T^t\cdot M(\tau)\right) \\ 0 & T^{-1} \end{pmatrix}. 
    \]
    Moreover, the action of the group $\Sp_{2g}(S)$ on $\bb H_g$ has the following geometric meaning: if we consider $\bb H_g$ as a moduli space of $g$-dimensional, principally polarized complex abelian varieties with symplectic basis, then two points $y = [A_y]$ and $z = [A_z] \in \bb H_g$ are in the same $\Sp_{2g}(S)$-orbit if and only if there exists an isogeny $\phi \colon A_y \to A_z$ that preserves the polarizations up to a product of powers of $p$ and $q$. Since the intersection
     \[
    f_\tau\left(\GL_g(\RR) \right) \cap \Sp_{2g}(S) = f_\tau\left(\GL_g^\tau(S)\right)
    \]
    equals the subgroup of $\Sp_{2g}(S)$ that preserves the locus
    \[
    \rho_\tau(\rm H_g) = \left\{
    \frac{1}{2} \cdot M(\tau) + \rm{i} N\right\} \subset \bb H_g
    \]
 of real abelian varieties of type $\tau$, this concludes Step 2. 
    \item[Step 3:] \emph{For any $\tau \in \mr T(g)$, the subgroup $\GL_g^\tau(S)\subset \GL_g(\RR)$ is dense in the analytic topology.}


Define  \[\SL_g^\tau(S)  = \SL_g(S) \cap \GL_g^\tau(S) =  \{T \in \SL_g(S) \colon T^t \cdot M(\tau) \cdot T \equiv M(\tau) \bmod 2\}.\] 
We claim that 
\begin{align}\label{inclusionsgroups}
\SL_g(\RR)  = \overline{\SL_g(S)} =\overline{\SL_g^\tau(S)} \subset \overline{\GL_g^\tau(S)} \subset \GL_g(\RR).
\end{align}
Indeed, this follows from the following two statements:
\begin{enumerate}
\item \label{proofparttwo} The closure of $\SL_g(S)$ in $\GL_g(\RR)$ is $\SL_g(\RR)$. 
\item \label{proofpartone} The subgroup $\SL_g^\tau(S) \subset \SL_g(S)$ has finite index. 
\end{enumerate}
\noindent
To prove (\hyperlink{proofparttwo}{a}), observe that the subgroup $\SL_g(\RR) \subset \GL_g(\RR)$ is closed, which implies that the closure of $\SL_g(S)$ in $\GL_g(\RR)$ equals the closure of $\SL_g(S)$ in $\SL_g(\RR)$. Thus, (\hyperlink{proofparttwo}{a}) follows from the density of $\SL_g(S)$ in $\SL_g(\RR)$, which is true by strong approximation; see \cite[Lemma 7.14]{thesis-degaayfortman} for the precise argument. As for (\hyperlink{proofpartone}{b}), the group $\SL_g(S)$ acts on $\rm M_g(S/2S)$ via $N \mapsto T^t \cdot N \cdot T \bmod 2$ for $T \in \SL_g(S)$, so there is an injection
\[
\SL_g^\tau(S) \setminus \SL_g(S) \hookrightarrow \rm M_g(S/2S) = \rm M_g(\ZZ/2), \quad T \mapsto T^t \cdot M(\tau) \cdot T \mod 2. 
\]
Statement (\hyperlink{proofpartone}{b}) 
implies that the index of $\overline{\SL_g^\tau(S)} \subset \overline{\SL_g(S)}$ is finite. e have $\overline{\SL_g(S)} = \SL_g(\RR)$ by (\hyperlink{proofparttwo}{a}), thus $\overline{\SL_g^\tau(S)}  \subset \SL_g(\RR)$ is a closed subgroup of finite index, hence open. Therefore $\overline{\SL_g^\tau(S)}  = \SL_g(\RR)$ by connectivity of $\SL_g(\RR)$, proving Claim (\ref{inclusionsgroups}). 

Write $G = \GL_g^\tau(S)$. If $H \subset \GL_g(\RR)$ is any Lie subgroup such that $\SL_g(\RR) \subset H$, then $H = \det^{-1}(\det(H))$. 
Consequently, using (\ref{inclusionsgroups}), we obtain:
\begin{align}\label{determinantdetermins}
\overline G = {\det}^{-1}(\det(\overline G)) \subset \GL_g(\RR).
\end{align}
The equality (\ref{determinantdetermins}) implies that Step 3 reduces to the equality ${\det}^{-1}(\det(\overline G)) = \GL_g(\RR)$. This follows from the equality $\det(\overline G) = \RR^\ast$, which we now prove. The morphism $\det\colon \GL_g(\RR) \to \RR^\ast$ is open because its differential at the identity matrix $I_g \in \GL_g(\RR)$ is the trace homomorphism $\rm M_g(\RR) \to \RR$. Writing
\[
\GL_g(\RR) = {\det}^{-1}(\RR^\ast) = {\det}^{-1}(\det(\overline G) \sqcup \det(\overline G)^c) =  \overline{G} \sqcup {\det}^{-1}\left(\det(\overline G)^c \right),
\]
it follows that 
$\det(\overline G)$ is closed in $\RR^\ast$. We conclude that 
$
\overline{\det(G)} \subset \overline{\det(\overline G)} = \det(\overline G)$. Thus, to show that $\det(\overline G) = \RR^\ast$, it suffices to show that $\det(G)$ is dense in $\RR^\ast$. The homomorphism $\det \colon \GL_g^\tau(S) \to S^\ast$ is surjective because it admits the section 
\[
S^\ast \to \GL_g^\tau(S), \quad x \mapsto \begin{pmatrix} x & 0 \\ 0 & I_{g-1} \end{pmatrix}.
\]
Therefore, \[\det(G) = \det(\GL_g^\tau(S) ) = S^\ast = \set{\pm p^nq^m\colon n,m \in \ZZ},\] and it remains to prove that the latter is dense in $\RR^\ast$. This holds, since $S^\ast_{>0} =  \set{p^nq^m\colon n,m \in \ZZ}$ is dense in $\RR_{>0}$ because $\log(S^\ast_{>0}) = \ZZ \log(p) + \ZZ\log(q)$ is dense in $\RR$; the latter follows from the fact that $\log(S^\ast_{>0})$ is not a cyclic subgroup of $\RR$. 

\item[Step 4:] \emph{Finish the proof.} 
Let $x = [(A, \lambda)] \in \va{\ca A_g(\RR)}$, and let $\tau \in \mr T(g)$ be the type of the principally polarized real abelian variety $(A,\lambda)$. Lift $x$ to a point $y \in \rm H_g$. By Step 3, we know that the orbit $\GL_g^\tau(S)\cdot y$ is dense in $\GL_g(\RR) \cdot y = \rm H_g$. Consequently, the image of $\GL_g^\tau(S)\cdot y$  under the projection $\rm H_g \to\va{\ca A_g(\RR)}^\tau$ is dense in $\va{\ca A_g(\RR)}^\tau$. By Step 2, this image is precisely $\ca G_{p,q}(x)$. Thus $\ca G_{p,q}(x)$ is dense in $\va{\ca A_g(\RR)}^\tau$ as desired. 
\end{enumerate}
\end{proof}


\section{Principal polarizations in real isogeny classes} \label{sec:principalpolarizations}

The goal of this section is to prove the following:

\begin{theorem} \label{principalisogeny}
Let $(A, \lambda_A)$ be a polarized abelian variety over $\RR$. Then there exists a principally polarized abelian variety $(B, \lambda_B)$ over $\RR$ and an isogeny 
\[
\phi \colon A \to B \quad \textnormal{ \emph{such that} } \quad \phi^\ast(\lambda_B) = \lambda_A. 
\]
\end{theorem}

\begin{proof}
Let $K \subset A(\CC)$ be the kernel of the analytified polarization $\lambda_A \colon A(\CC) \to \wh A(\CC)$. Then $K$ is a finite group of order $d^2$, where $d^2$ is the degree of $\lambda_A$, such that the real structure $\sigma \colon A(\CC) \to A(\CC)$ restricts to an involution
\[
\sigma \colon K \to K. 
\]
We may assume that $K \neq (0)$. Let $p$ be any prime number that divides the order of $K$. We claim that there exists a subgroup $K_1 \subset K$ of order $p$ such that $\sigma(K_1) = K_1$. To see this, let 
$H[p] \subset K$ be the $p$-torsion subgroup of $K$. Then $H[p]$ is preserved by $\sigma$, so that $H[p]$ is an $\FF_p$-vector space of finite rank equipped with a linear involution $\sigma$. Therefore, $H[p]$ contains a one-dimensional $\FF_p$-subspace $K_1$ preserved by $\sigma$, which proves our claim. 

The group $K_1 \subset A(\CC)$ descends to a finite subgroup scheme $K_1 \subset A$ over $\RR$; define $A_1$ to be the abelian variety $A/K_1$ over $\RR$. Let $\Lambda = \rm H_1(A(\CC),\ZZ)$ and $M = \rm H_1(A_1(\CC), \ZZ)$; the projection $A \to A_1$ induces an exact sequence 
\[
0 \to \Lambda \to M \to K_1 \to 0.
\]
Let $E \colon \Lambda \times \Lambda \to \ZZ$ be the alternating form attached to the polarization $\lambda_A$ of $A$. Since $M / \Lambda = K_1 \subset K = \Lambda^\vee/\Lambda$, we have inclusions
\[
\Lambda \subset M \subset \Lambda^\vee, \quad\textnormal{ where } \quad \Lambda^\vee = \set{x \in \Lambda \otimes \QQ \mid E(x, \Lambda) \subset \Lambda}. 
\]
The $\ZZ$-valued alternating form $E$ on the lattice $\Lambda$ gives rise to a bilinear form
\[
\overline{E} \colon \Lambda^\vee/\Lambda \times \Lambda^\vee/\Lambda \to \QQ/\ZZ
\]which vanishes on $M/\Lambda$ because $M/\Lambda \cong \ZZ/p$ and $\overline{E}$ is alternating. Thus $E \colon \Lambda^\vee \times \Lambda^\vee \to \QQ$ restricts to an integer-valued form $E_1$ on $M$. The latter induces a polarization $\lambda_{A_1} \colon A_1 \to \wh{A}_1$ that makes the following diagram commute:
\[
\xymatrix{
A \ar[d]^{\lambda_A} \ar[r]^{\pi}& A_1 \ar[d]^{\lambda_{A_1}} \\
\wh A & \wh A_1 \ar[l]_{\wh \pi}.
}
\]
Here $\pi$ is the quotient map $A \to A_1$ and $\wh \pi$ its dual. Since the degree of an isogeny is multiplicative in compositions, and $\deg(\pi) = p$, we have
\[
p^2 \cdot \deg(\lambda_{A_1}) = \deg(\pi)^2\cdot \deg(\lambda_{A_1}) = \deg\left( \lambda_A \right) = d^2. 
\]
If $d = p$, we are finished -- otherwise, we repeat the above procedure until the real abelian variety $A_n = A_{n-1}/K_{n-1}$ becomes principally polarized. \end{proof}

\section{Integral Hodge classes modulo torsion}

\subsection{The Fourier transform}\label{subsec:fourier} Let $A$ be a real abelian variety of dimension $g$, and consider the Poincar\'e bundle $\ca P_A$ on $A \times \wh A$. Let $a_1, \dotsc, a_{2g}$ be integers. The Chern character
\[
\ch(\ca P_{A_\CC}) = \exp(c_1(\ca P_{A_\CC})) \in  \rm H^{2\bullet}(A(\CC) \times \wh A(\CC), \ZZ(\bullet))
\]
defines the \emph{Fourier transform}
\begin{align}\label{fourierone}
\mr F_A \colon  \bigoplus_{i \in \ZZ_{\geq 0}} \rm H^i(A(\CC),\ZZ(a_i)) \to \bigoplus_{i \in \ZZ_{\geq 0}}\rm H^{i}(\wh A(\CC),\ZZ(a_{2g-i}- g + i)).
\end{align}
It is defined as $\mr F_A(x) = \pi_{2,\ast}( \ch(\ca P_A) \cdot \pi_1^\ast(x))$, where $\pi_i$ is the projection of $A \times \wh A$ onto the $i$-th factor. By \cite{beauvillefourier}, the map (\ref{fourierone}) is an isomorphism, inducing isomorphisms
\begin{align}\label{eq:fourier}
\mr F_A \colon \rm H^i(A(\CC),\ZZ(a_i)) \xrightarrow{\sim} \rm H^{2g-i}(\wh A(\CC),\ZZ(a_i+g-i)). 
\end{align}
Since $\ch(\ca P_{A_\CC})$ is fixed by $G$, these maps are isomorphisms of $G$-modules. 

\subsection{Divisors} \label{subsec:divisors}Let $A$ be an abelian variety over $\RR$. 
Both homomorphisms in the following composition are surjective:
\begin{align}
\label{divisorsurjectivity}
\CH^1(A) \to \Hdg_G^2(A(\CC),\ZZ(1)) \to \Hdg^2(A(\CC),\ZZ(1))^G. 
\end{align}
Indeed, the first map is surjective by the real integral Hodge conjecture for divisors (see \cite[Proposition 2.8]{BW20}), and the second by the degeneration of the Hochschild-Serre spectral sequence (see Section \ref{subsec:hochschild}). 

\subsection{The real integral Hodge conjecture for one-cycles modulo torsion} The goal of this section is to provide an application of Theorems \ref{density} and \ref{principalisogeny} combined: the real integral Hodge conjecture for one-cycles modulo torsion follows, in some cases, from the real integral Hodge conjecture for divisors modulo torsion (c.f.\ Section \ref{subsec:divisors}). For this, we use Fourier transforms on Betti cohomology of real abelian varieties (c.f.\ Section \ref{subsec:fourier}) 
and their lifts to integral Chow groups, see \cite{beckmanndegaayfortman, moonenpolishchuk}. The theory in \cite[Section 3]{beckmanndegaayfortman} was developed for abelian varieties over a general field $k$ -- to apply it, we take $k = \RR$. 

\begin{definition}
Let $A$ be a real abelian variety, and $k$ a non-negative integer. An element $\alpha \in \Hdg^{2k}(A(\CC),\ZZ(k))^G$ is called \emph{algebraic} if it is in the image of 
\[
\CH^k(A) \to  \Hdg^{2k}(A(\CC),\ZZ(k))^G. 
\]
\end{definition}
\noindent
By the main theorem of \cite{beckmanndegaayfortman}, the integral Hodge conjecture for one-cycles on a fixed principally polarized abelian variety $(A, \theta)$ over $\CC$ is equivalent to the algebraicity of the minimal class $\gamma_\theta = \theta^{g-1}/(g-1)!$ on $A$ (see \emph{loc.cit.}, Theorem 1.1). Moreover, Grabowski reduced the integral Hodge conjecture for one-cycles for every complex abelian variety of dimension $g$ to the algebraicity of $\gamma_\theta$ for every principally polarized abelian variety of dimension $g$, see \cite{grabowski}. We have the following real analogue of these results:


\begin{theorem} \label{grabowski}
Fix a positive integer $g$. Let $(A, \theta)$ be a principally polarized abelian variety of dimension $g$ over $\RR$. The following are equivalent: 
\begin{enumerate}
\item \label{realmodtors} The real abelian variety $A$ satisfies the real integral Hodge conjecture for one-cycles modulo torsion. 
\item \label{realminalg} The minimal class 
\[
\gamma_\theta = \frac{\theta^{g-1}}{(g-1)!} \in \Hdg^{2g-2}(A(\CC), \ZZ(g-1))^G \quad \text{ is algebraic.}
\]
\item \label{realchernalg} The Chern character 
\[
\ch(\ca P_{A_\CC}) = \exp(c_1(\ca P_{A_{\CC}})) \in \Hdg^{2\bullet}(A(\CC ) \times \wh A(\CC), \ZZ(\bullet))^G
\quad \text{ is algebraic.}
\]
\end{enumerate}
Moreover, if the real integral Hodge conjecture for one-cycles modulo torsion holds for every principally polarized abelian variety of dimension $g$ over $\RR$, then it holds for every abelian variety of dimension $g$ over $\RR$. 
\end{theorem}

\begin{proof}
The direction \ref{realmodtors} $\implies$\ref{realminalg} is trivial; let us assume that \ref{realminalg} holds. Let $(A, \theta)$ be a principally polarized abelian variety of dimension $g$ over $\RR$, and suppose that $\gamma_\theta$ 
is algebraic. 
By \cite[Proposition 3.11]{beckmanndegaayfortman}, the abelian variety $A$ admits a motivic integral Fourier transform up to homology, see \cite[Definition 3.1]{beckmanndegaayfortman}. This means the following. Let $\ell$ be a prime number. 
There exists a cycle 
\[
\Gamma \in \CH(A \times \wh A) \quad \textnormal{ such that } \quad [\Gamma_\CC] = \ch(\ca P_{A_{\CC}}) \quad \in \quad \rm H^{2\bullet}_{\et}(A_{\CC} \times \wh A_{\CC}, \ZZ_\ell(\bullet)).
\]
As a consequence, $[\Gamma_\CC] = \ch(\ca P_{A_\CC}) \in \rm H^{2\bullet}(A(\CC) \times \wh A(\CC), \ZZ(\bullet))^G$, which proves \ref{realchernalg}. 

Let us now assume that \ref{realchernalg} holds, and let $\Gamma \in \CH(A \times \wh A)$ be a cycle that induces the class $\ch(\ca P_{A_\CC})$ in Betti cohomology. The correspondence $\Gamma$ defines a group homomorphism $\Gamma_\ast \colon \CH^\bullet(A) \to \CH^\bullet(\wh A)$ such that the following diagram commutes:
{\small
\begin{align*}
\xymatrixcolsep{1pc}
\xymatrix{
\CH^1(A) \ar[d] \ar[r]  
& \CH^\bullet(A) \ar[d] \ar[r]^{\Gamma_\ast} &
 \CH^\bullet(\wh A) \ar[d] \ar[r] & 
 \CH_1(\wh A) \ar[d] \\
\Hdg^{2}(A(\CC),\ZZ(1))^G \ar[r] & \Hdg^{2\bullet}(A(\CC), \ZZ(\bullet))^G \ar[r]_\sim^{\mr F_A} & \Hdg^{2\bullet}(\wh A(\CC),\ZZ(\bullet))^G  \ar[r]& \Hdg^{2k}(\wh A(\CC),\ZZ(k))^G. 
}
\end{align*}
}
Here $k = g-1$, the composition on the bottom row is an isomorphism (see (\ref{eq:fourier}) in Section \ref{subsec:fourier}), and the left vertical arrow is surjective (see (\ref{divisorsurjectivity}) in Section \ref{subsec:divisors}). Therefore, the right vertical arrow is surjective, which implies \ref{realmodtors}. 

Next, suppose that \ref{realmodtors} holds for every principally polarized real abelian variety of dimension $g$, and let $A$ be any real abelian variety of dimension $g$. We would like to show that $A$ satisfies the real integral Hodge conjecture for one-cycles modulo torsion. The isomorphism (\ref{eq:fourier}) induces an isomorphism 
\[
\mr F_A \colon \Hdg^2(A(\CC), \ZZ(1))^G \xrightarrow{\sim} \Hdg^{2g-2}(\wh A(\CC), \ZZ(g-1))^G.
\]
Since $\Hdg^2(A(\CC), \ZZ(1))^G$ is algebraic by Section \ref{subsec:divisors}, it suffices to show that $\mr F_A([L_\CC])$ is algebraic for every line bundle $L$ on $A$. By \cite[II, Exercice 7.5]{HAG}, there is an ample line bundle $M$ on $A$ such that $L \otimes M^{\otimes n}$ is ample for $n \gg 0$; we may thus assume that $L$ is ample. By Theorem \ref{principalisogeny}, there is a principally polarized abelian variety $(B, \lambda)$, and an isogeny 
\[
\phi \colon A \to B
\]
 such that, if $\theta \in \NS(B_\CC)^G = \Hdg^{2}(B(\CC), \ZZ(1))^G$ is the class corresponding to the principal polarization $\lambda \colon B \to \wh B$, then 
\[
\phi^\ast(\theta) = [L_\CC] \in \Hdg^2(A(\CC), \ZZ(1))^G. 
\]
On the other hand, the following diagram commutes by \cite[Proposition 3]{beauvillefourier}:
\[
\xymatrixcolsep{5pc}
\xymatrix{
\Hdg^2(A(\CC), \ZZ(1))^G \ar[r]^{\mr F_A}          & \Hdg^{2g-2}(\wh A(\CC), \ZZ(g-1))^G          \\
\Hdg^2(B(\CC), \ZZ(1))^G \ar[r]^{\mr F_B}\ar[u]_{\phi^\ast}  & \Hdg^{2g-2}(\wh B(\CC), \ZZ(g-1))^G \ar[u]_{\hat{\phi}_\ast}. 
}
\]
Moreover, by \cite[Proposition 5]{beauvillefourier}, we have
\[
\mr F_B(\theta) = (-1)^{g-1}\cdot \frac{\hat{\theta}^{g-1}}{(g-1)!} \in \rm H^{2g-2}(\wh B(\CC), \ZZ(g-1))^G,
\]
where $\hat{\theta} \in \Hdg^2(\wh B(\CC), \ZZ(1))^G$ denotes the dual polarization class. We conclude that $\mr F_B(\theta)$ is algebraic, so that $\mr F_A([L_\CC]) = \wh{\phi}_\ast(\mr F_B(\theta))$ is algebraic as well. \end{proof}


\begin{corollary} \label{IHCmodulotorsionforjacobians} Let $C_1, \dotsc, C_n$ be smooth projective geometrically integral curves over $\RR$ such that $C_i(\RR)\neq \emptyset$ for each $i$. The real abelian variety $A = J(C_1) \times \cdots \times J(C_n)$ satisfies the real integral Hodge conjecture for one-cycles modulo torsion. 
\end{corollary}
\begin{proof}
The minimal class on a product of principally polarized abelian varieties over $\RR$ is the sum of the pull-backs of the minimal classes on the factors, so by Theorem \ref{grabowski}, it suffices to treat the case $n = 1$. For a real algebraic curve $C$ whose real locus is non-empty, any Abel-Jacobi map gives an embedding of real varieties $\iota \colon C \hookrightarrow J(C)$. By Poincar\'e's formula, one has 
\[
[\iota(C)_\CC] = \frac{\theta^{g-1}}{(g-1)!} \in \Hdg^{2g-2}(J(C)(\CC), \ZZ(g-1))^G,
\]
where the class on the right hand side of the equality is the minimal cohomology class $\gamma_\theta$ of $J(C)$. Thus $\gamma_\theta$ is algebraic, so we are done by Theorem \ref{grabowski}. 
\end{proof}

\subsection{Integral Hodge classes modulo torsion on real abelian threefolds}\label{sec:modtors}
We define a \emph{real algebraic curve} as a smooth projective geometrically connected curve over $\RR$. Let $\ca M_3(\RR)$ be the moduli space of real algebraic curves of genus three (see \cite{seppalasilhol} and \cite[Theorem 8.2]{MR4479836}), and consider the Torelli map 
\begin{align*}
t \colon \ca M_3(\RR) \to \ca A_3(\RR). 
\end{align*}
Let $\ca N_3(\RR) \subset \ca M_3(\RR)$ be the non-hyperelliptic locus. 

\begin{lemma}\label{nonhyperelliptic-open}The subset $t(\ca N_3(\RR))$ is open in $\ca A_3(\RR)$. \end{lemma}
\begin{proof}
On the level of stacks, $\ca T \colon \ca M_3 \to \ca A_3$ is an open immersion when restricted to the non-hyperelliptic locus $\ca N_3 \subset \ca M_3$. 
Moreover, for any algebraic stack $\mr X$ of finite type over $\RR$, the set $|\mr X(\RR)|$ of isomorphism classes of $\RR$-points of $\mr X$ admits a \emph{real-analytic topology} (see \cite[Definition 7.5]{MR4479836}). For this topology, the subset
$|\ca T(\ca N_3)(\RR)| \subset |\ca A_3(\RR)|$ is open by \cite[Corollary 2.8.2]{thesis-degaayfortman}. By Remark \ref{remark:bijection=homeo}, we are done. 
\end{proof}

\begin{lemma} \label{lemma:opensubsetofnonhyp}
Every connected component of the moduli space $\ca A_3(\RR)$ contains a non-empty open subset of non-hyperelliptic curves of genus three with non-empty real locus. 
\end{lemma}
\begin{proof}
Let $\ca A_3(\RR)^\tau$ be a connected component of $\ca A_3(\RR)$. By \cite[page 182]{grossharris}, there is a (unique) connected component $\ca M_3(\RR)^\tau$ of $\ca M_3(\RR)$ that satisfies the following two conditions: $C(\RR)$ is not empty for each $[C] \in \ca M_3(\RR)^\tau$, and
\[
t \left(\ca M_3(\RR)^\tau\right) \subset \ca A_3(\RR)^\tau. 
\]
Now $\ca M_3(\RR)^\tau$ contains a component $\ca N_3(\RR)^\tau$ of $\ca N_3(\RR)$ by \cite[Proposition 3.1]{grossharris} and the table on page 174 of \emph{loc.cit.} Moreover, $\ca N_3(\RR)^\tau$ is open in $\ca A_3(\RR)$ by Lemma \ref{nonhyperelliptic-open}. 
\end{proof}

\begin{proof}[Proof of Corollary \ref{usefulcorollary}] 
This follows directly from Theorem \ref{density} and Lemma \ref{lemma:opensubsetofnonhyp}.\end{proof}
\noindent
We are now in the position to prove Theorem \ref{theorem1}. 
\begin{proof}[Proof of Theorem \ref{theorem1}]  By Theorem \ref{grabowski}, it suffices to show that for any principally polarized abelian threefold $(A, \theta)$ over $\RR$, the class $\gamma_\theta = \theta^2/2 \in \Hdg^4(A(\CC),\ZZ(2))^G$ is algebraic. Let us prove this. Let $p$ and $q$ be two distinct odd prime numbers. By Corollary \ref{usefulcorollary}, there exists a real algebraic curve $C$ of genus three such that $C(\RR) \neq \emptyset$ together with an isogeny 
\[
\phi \colon A \to J(C)
\]
such that $\phi^\ast(\lambda_C) = p^n q^m\cdot \lambda_A$ for some non-negative integers $n$ and $m$. Here $\lambda_C$ denotes the canonical polarization on $J(C)$. Let $\theta_C \in \Hdg^2(J(C)(\CC),\ZZ(1))^G$ be the corresponding cohomology class; then $\phi^\ast(\theta_C) = p^nq^m \cdot \theta$. Since $C(\RR) \neq \emptyset$, the minimal class $\theta^2_C/2$ on $J(C)$ is algebraic; see Corollary \ref{IHCmodulotorsionforjacobians}. Therefore, the class 
\[
p^{2n}q^{2m} \cdot \theta^2/2 = \phi^\ast(\theta^2_C/2) \in \Hdg^{4}(A(\CC), \ZZ(2))^G 
\] is algebraic. We are done because $\theta^2$ is algebraic and the integer $p^nq^m$ is odd. 
\end{proof}

\begin{proof}[Proof of Corollary \ref{IHCforconnected}] For each $i \in \{1, \dotsc, 6\}$, cup-product defines a canonical isomorphism of $G$-modules:
\[
\bigwedge^i\rm H^1(A(\CC), \ZZ) \xrightarrow{\sim} \rm H^i(A(\CC), \ZZ). 
\]
This allows us to calculate the $G$-module structure on $\rm H^i(A(\CC), \ZZ)$, for $\rm H^1(A(\CC), \ZZ) \cong \ZZ[G]^3$ as $G$-modules \cite[Example 3.1]{vanhamel}. It turns out that the group $\rm H^p(G, \rm H^q(A(\CC), \ZZ))$ vanishes whenever $p + q = 4$ and $p > 0$. Therefore, \[\Hdg^4_G(A(\CC), \ZZ(2))_0 = \Hdg^4(A(\CC), \ZZ(2))^G\] by Section \ref{subsec:hochschild} and Lemma \ref{lemma:important0}. By Theorem \ref{theorem1}, we are done.
\end{proof}

\section{Torsion cohomology classes on real abelian varieties}

\subsection{Isogenies and torsion cohomology classes} Let $X$ be a smooth projective variety over $\RR$, and let $k$ be any non-negative integer. 

\begin{definition} \label{isogenydef:subgroups}
For a subgroup $K \subset \rm H^{2k}_G(X(\CC), \ZZ(k))$, let
$
K_0 = K \cap \rm H^{2k}_G(X(\CC),\ZZ(k))_0$. In case the Hochschild-Serre spectral sequence (\ref{hochschild}) degenerates, we define, for $p > 0$,
\begin{align*}
\rm H^p(G, \rm H^{2k-p}(X(\CC),\ZZ(k)))_0 = \Ima(F^p\rm H^{2k}_G(X(\CC),\ZZ(k))_0 \to \rm H^p(G, \rm H^{2k-p}(X(\CC),\ZZ(k)))). 
\end{align*}
A subgroup $K \subset \rm H^{2k}_G(X(\CC),\ZZ(k))$ is \emph{algebraic} if every element of $K$ is in the image of the cycle class map (\ref{realcycleclassmap}). If the Hochschild-Serre spectral sequence (\ref{hochschild}) degenerates, we call a subgroup $L \subset \rm H^p(G, \rm H^{2k-p}(X(\CC), \ZZ(k)))$ \emph{algebraic} if $L$ is the image of an algebraic subgroup $K \subset F^p\rm H^{2k}_G(X(\CC),\ZZ(k))$ under $F^p\rm H^{2k}_G(X(\CC),\ZZ(k)) \to  \rm H^p(G, \rm H^{2k-p}(X(\CC), \ZZ(k)))$. 
\end{definition}
\noindent
The following lemma provides the key to our proof of Theorem \ref{reduction}. 

\begin{lemma} \label{algebraicity}
Let $A$ and $B$ be real abelian varieties, and $\phi$ an isogeny
$
A \to B
$
of odd degree. For all integers $p, k, i$ with $p>0$, the pull-back $\phi^\ast$ defines an isomorphism 
\begin{align} \label{filtrationisomorphism}
F^p\rm H^{k}_G(B(\CC), \ZZ(i))_0 \xrightarrow{\sim} F^p\rm H^{k}_G(A(\CC), \ZZ(i))_0. 
\end{align}
If $k \in 2 \ZZ_{\geq 0}$ and $i = k/2$, then $\phi^\ast$ restricts to an isomorphism between the subgroups of classes that satisfy the topological condition (\ref{topologicalcondition}). In this case, the left-hand side of (\ref{filtrationisomorphism}) is algebraic (in the sense of Definition \ref{isogenydef:subgroups}) if and only if the right-hand side of (\ref{filtrationisomorphism}) is algebraic. 
\end{lemma}
\begin{proof}
Let us first prove the lemma in the case $A = B$ and $\phi = [d]$ for some integer $d \in \ZZ$ with $d \not \equiv 0 \bmod 2$. 
Write $\rm H = \rm H^{k}_G(A(\CC), \ZZ(i))$. The Hochschild-Serre spectral sequence (\ref{hochschild}) degenerates, so $F^1\rm H = \rm H[2]$ and for each $p \in \{0,\dotsc,k\}$ there is an exact sequence
\[
0 \to F^{p+1}\rm H \to F^p \rm H \to \rm H^p(G, \rm H^{k-p}(A(\CC),\ZZ(i))) \to 0. 
\]
We have $F^{k + 1}\rm H = (0)$, and for each $p \in \{1,\dotsc,k\}$, the morphism \[[d]^\ast \colon \rm H^{k-p}(A(\CC),\ZZ(i)) \to \rm H^{k-p}(A(\CC),\ZZ(i))\] induces the identity on $\rm H^p(G, \rm H^{k-p}(A(\CC),\ZZ(i)))$. By the snake lemma and descending induction on $p$, we find that $[d]^\ast \colon \rm H \to \rm H$ restricts to an isomorphism on $F^p\rm H$ for each $p>0$. 

Write $\rm M = \rm H^{2k}_G(A(\CC), \ZZ(k))$. By \cite[\S1.6.4]{BW20}, the class $[d]^\ast(\alpha)$ satisfies the topological condition (\ref{topologicalcondition}) whenever $\alpha \in \rm M$ the topological condition (\ref{topologicalcondition}), thus $[d]^\ast$ induces an embedding $F^p\rm M_0 \to F^p\rm M_0$, which is an isomorphism since $F^p\rm M_0$ is finite. For similar reasons, $[d]^\ast \colon F^p\rm M_{\tn{alg}}  \to  F^p\rm M_{\tn{alg}} $ is an isomorphism, where $F^p\rm M_{\tn{alg}}  \subset F^p\rm M$ is the subgroup of algebraic elements. In the general case, let $\psi \colon B \to A$ be an isogeny such that $\psi \circ \phi = [d]_A$ and $\phi \circ \psi = [d]_B$, where $d$ is the degree of the isogeny $\phi$. The equalities
\[
\psi^\ast \circ \phi^\ast = [d]_A^\ast \quad \textnormal{ and } \quad \psi^\ast \circ \phi^\ast = [d]_B^\ast 
\]
together with the conclusion of the lemma for $[d]_A$ and $[d]_B$ readily imply the result. 
\end{proof}


\subsection{Reduction to the Jacobian case} \begin{proof}[Proof of Theorem \ref{reduction}]
\
Let $\ca A_3(\RR)^+$ be a connected component of $\ca A_3(\RR)$, 
and consider a point $x = [(A, \lambda)] \in \ca A_3(\RR)^+$. By Corollary \ref{usefulcorollary}, there is a real algebraic curve $C$ with non-empty real locus together with an isogeny 
\[
\phi \colon A \to J(C)
\]
that preserves the polarizations up to an odd positive integer. By \cite[page 180]{grossharris}, we have $[J(C)] \in \ca A_3(\RR)^+$ for the moduli point $[J(C)]$ of the principally polarized Jacobian $J(C)$ of $C$. The following sequence is exact (see Lemma \ref{lemma:important0}): 
\[
0 \to \rm H^4_G(A(\CC),\ZZ(2))_0[2] \to \Hdg^{4}_G(A(\CC),\ZZ(2))_0 \to \Hdg^4(A(\CC),\ZZ(2))^G \to 0.  
\]
By Theorem \ref{theorem1}, we know that $\Hdg^4(A(\CC),\ZZ(2))^G$ is generated by classes of one-cycles on $A$. Therefore, $A$ satisfies the real integral Hodge conjecture if and only if the group $\rm H^4_G(A(\CC),\ZZ(2))_0[2]$ is algebraic. By Lemma \ref{algebraicity}, this is equivalent to the algebraicity of $\rm H^4_G(J(C)(\CC),\ZZ(2))_0[2]$. 
\end{proof}

\subsection{Analysis of the Hochschild-Serre filtration} \label{subsec:analysis}Let $A$ be an abelian variety over $\RR$, define $\rm H = \rm H^4_G(A(\CC),\ZZ(2))$, and consider the Hochschild-Serre filtration (c.f.\ Section \ref{subsec:hochschild}):
\begin{align}\label{serrefiltration}
0 \subset \rm H^4(G,\rm H^0(A(\CC),\ZZ(2))) = F^4\rm H \subset F^3\rm H \subset F^2\rm H \subset F^1\rm H = \rm H[2] \subset \rm H. 
\end{align}
The pull-back $\pi^\ast$ along the structural morphism $\pi \colon A \to \Spec(\RR)$ defines a section 
\[
\ZZ/2 =  \rm H^4_G(\{x\}, \ZZ(2)) = \rm H^4(G, \rm H^0(\{x\}, \ZZ(2))) \to \rm H^4(G, \rm H^0(A(\CC),\ZZ(2))) \subset \rm H
\]
of the restriction $\rm H \to \rm H^4_G(\{x\},\ZZ(2)) = \ZZ/2$ to the equivariant cohomology of any $\RR$-point $x \in X(\RR)$. By Section \ref{subsec:threefolds}, this implies that $F^4\rm H_0 = (0)$, so that the intersection of (\ref{serrefiltration}) with the group of classes satisfying the topological condition (\ref{topologicalcondition}) becomes
\begin{align}\label{topologically-distinguished-filtration}
0 \subset F^3\rm H_0 \subset F^2\rm H_0 \subset F^1\rm H_0 \subset \rm H_0. 
\end{align}
Continue to consider our abelian variety $A$ over $\RR$. 
\emph{Assume that there exists a cycle}
\[
\Gamma \in \CH(A \times \wh A) \quad \textnormal{ \emph{such that} } \quad [\Gamma_{\CC}] = \ch(\ca P_{A_\CC})  \in  
\rm H^{2\bullet}(A(\CC) \times \wh A(\CC),\ZZ(\bullet))^G. 
\]
Under this assumtion, we may define a homomorphism $\Gamma_\ast$ as in the following diagram:
\begin{align} \label{fundamentalcorrespondence}
\begin{split}
\xymatrixcolsep{5pc}
\xymatrix{
\rm H^{2\bullet}_G(A(\CC),\ZZ(\bullet)) \ar[d]^{\Gamma_\ast}\ar[r]^{\pi_1^\ast} & \rm H^{2\bullet}_G(A(\CC) \times \wh A(\CC),\ZZ(\bullet))  \ar[d]^{\left([\Gamma] \cdot -\right)}  \\
\rm H^{2\bullet}_G(\wh A(\CC),\ZZ(\bullet)) & \rm H^{2\bullet}_G(A(\CC) \times \wh A(\CC),\ZZ(\bullet))  
\ar[l]^{\pi_{2,\ast}}. 
}
\end{split}
\end{align}
Then $\Gamma_\ast$ preserves the classes that satisfy the topological condition (\ref{topologicalcondition}) by \cite[Theorem 1.21]{BW20}. Define 
\begin{align}\label{fundamental-ij}
\Gamma_{\ast}^{i,j} \colon \rm H^{2i}_G(A(\CC),\ZZ(i)) \to \rm H^{2j}_G(\wh A(\CC),\ZZ(j)) \quad (i,j \in \ZZ_{\geq 0})
\end{align}
as the composition of $\Gamma_\ast$ with the natural inclusion and projection morphisms. 
\\
\\
With these definitions in place, we are in a position to prove:

\begin{proposition} \label{vanishing}
Let $A$  be an abelian variety over $\RR$. Then
\[
F^3\rm H^4_G(A(\CC),\ZZ(2))_0 = (0). 
\]
\end{proposition}
\begin{proof}
Let $g = \dim(A)$. By \cite[Chapter IV, Example 3.1]{vanhamel}, we have that $A(\CC)$ is $G$-equivariantly homeomorphic to the $g$-fold product of copies of the torus $\bb S \times \bb S$, where $\bb S \subset \CC^\ast$ is the unit circle and where $G$ acts on each copy $\bb S \times \bb S$ in one of the following ways: either $\sigma(x,y) = (x, \bar y)$ for $(x,y) \in \bb S \times \bb S$, or $\sigma(x,y) = (y,x)$ for $(x,y) \in \bb S\times \bb S$. 
In particular, there exist $g$ elliptic curves $E_i$ over $\RR$ and a $G$-equivariant homeomorphism 
\begin{align}\label{topologicalstatement}
A(\CC) \cong E_1(\CC) \times \cdots \times E_g(\CC). 
\end{align}
Since the conclusion of Proposition \ref{vanishing} is a statement that only depends on the structure of $A(\CC)$ as a topological $G$-space, we may therefore assume that $A$ is principally polarized by $\theta \in \Hdg^{2}(A(\CC),\ZZ(1))^G$ and that the following class is algebraic:
\[
\frac{\theta^{g-1}}{(g-1)!} \in \Hdg^{2g-2}(A(\CC), \ZZ(g-1))^G.
\]
By Theorem \ref{grabowski}, this means that the Chern character $\ch(\ca P_{A_\CC})$ is algebraic, which allows us to define a homomorphism $\Gamma_\ast \colon \rm H^{2\bullet}_G(A(\CC), \ZZ(\bullet)) \to \rm H^{2\bullet}_G(\wh A(\CC), \ZZ(\bullet))$ as in (\ref{fundamentalcorrespondence}), 
and homomorphisms $\Gamma_\ast^{ij}$ as in (\ref{fundamental-ij}). We have a commutative diagram: 
{\small
\begin{align*}
\xymatrixcolsep{2.6pc}
\xymatrix{
\rm H^4_G(A(\CC),\ZZ(2)) \ar[r]^{\Gamma_\ast^{4,2g+2}} & \rm H^{2g+2}_G(\wh A(\CC),\ZZ(g+1)) \ar[r]^{\Gamma_\ast^{2g+2,4}} & \rm H^4_G(A(\CC),\ZZ(2))  \ \\ 
F^3\rm H^4_G(A(\CC),\ZZ(2))  \ar[r]^{\Gamma_{\ast}^{4,2g+2}}\ar@{^{(}->}[u] \ar[d] & F^3\rm H^{2g+2}_G(\wh A(\CC),\ZZ(g+1)) \ar@{^{(}->}[u] \ar[d] \ar[r]^{\Gamma_\ast^{2g+2,4}} &F^3\rm H^4_G(A(\CC),\ZZ(2))   \ar@{^{(}->}[u] \ar[d]  \\
\rm H^3(G,\rm H^1(A(\CC),\ZZ(2))) \ar[r]^{\rm H^3(G, \mr F_A)\quad\;\;\;}& \rm H^3(G,\rm H^{2g-1}(\wh A(\CC),\ZZ(g+1)))
\ar[r]^{\quad {\scriptsize \rm H^3(G, \mr F_{\wh A})}}
 & \rm H^3(G, \rm H^1(A(\CC),\ZZ(2))). 
}
\end{align*}
}
\noindent
Here $\mr F_A \colon \rm H^1(A(\CC),\ZZ(2)) \to \rm H^{2g-1}(\wh A(\CC),\ZZ(g+1))$ is the Fourier transform considered in Section \ref{subsec:fourier}. This map is an isomorphism, with inverse \[(-1)^{1+g}\cdot\mr F_{\wh A} \colon \rm H^{2g-1}(\wh A(\CC),\ZZ(g+1)) \to \rm H^1(A(\CC),\ZZ(2)),\]
see \cite[Corollary 9.24]{huybrechtsfouriermukai}. Consequently, $\rm H^3(G, \mr F_A)$ is an isomorphism, with inverse $\rm H^3(G, \mr F_{\wh A})$. By the compatibility of $\Gamma_\ast$ with the topological condition (\ref{topologicalcondition}), we obtain an isomorphism
\begin{align}\label{h1h5}
\rm H^3(G, \mr F_A) \colon \rm H^3(G,\rm H^1(A(\CC),\ZZ(2)))_0 \xrightarrow{\sim} \rm H^3(G,\rm H^{2g-1}(\wh A(\CC),\ZZ(g+1)))_0. 
\end{align}
Because $\rm H^{2g+2}_G(\wh A(\CC),\ZZ(g+1))_0 = (0)$ by \cite[\S2.3.2]{BW20}, the group on the right of equation (\ref{h1h5}) vanishes. Therefore, the group on the left must be zero as well. 
\end{proof}
\noindent
In fact, the argument used in the above proposition can be generalized to prove the following lemma, which we record for later use:

\begin{lemma} \label{topologicalcompatibility}
Let $A$ be an abelian variety of dimension $g$ over $\RR$. Let $p$ and $q$ be non-negative integers such that $p + q = 2k \in 2 \ZZ_{\geq 0}$. The isomorphism on group cohomology 
\[
\rm H^p(G, \mr F_A) \colon \rm H^p(G, \rm H^q(A(\CC), \ZZ(k))) \to \rm H^p(G, \rm H^{2g-q}(\wh A(\CC), \ZZ(g+k-q)))
\]
identifies $\rm H^p(G, \rm H^q(A(\CC), \ZZ(k)))_0$ with $\rm H^p(G, \rm H^{2g-q}(\wh A(\CC), \ZZ(g+k-q)))_0$. 
\end{lemma}
\begin{proof}
This is a topological statement, so we may and do assume that $A$ is a product of elliptic curves (see (\ref{topologicalstatement})). In this case, we may lift $\rm H^p(G, \mr F_A)$ to an algebraic homomorphism
\[
\Gamma_\ast^{2k, 2g+2k-2q} \colon F^p\rm H^{2k}_G(A(\CC), \ZZ(k)) \to F^p\rm H^{2g+2k-2q}_G(\wh A(\CC), \ZZ(g+k-q))
\]
as above (see Section \ref{subsec:analysis}). We can do the same thing for $\rm H^p(G, \mr F_{\wh A})$; the lemma follows from arguments similar to those used to prove Proposition \ref{vanishing}. 
\end{proof}

\section{Abel-Jacobi maps over the real numbers} \label{section:AbelJacobi}

Let $X$ be a smooth projective variety over $\RR$. The goal of this section is to prove that the equivariant cycle class map and the real Abel-Jacobi map for $X$ are compatible; see Theorem \ref{abeljacobicommutativity} below for the precise statement. To prove this, we need Deligne cohomology and Deligne cycle class maps in the equivariant setting. Recall that Jannsen introduced $G$-equivariant Deligne cycle class maps for varieties over $\RR$ in \cite{janssen-Dhomology}. Since his results do not seem to directly imply Theorem \ref{abeljacobicommutativity}, we will take a somewhat different approach.

 Let $\sigma \colon X(\CC) \to X(\CC)$ be the canonical anti-holomorphic involution. Fix an integer $k$ with $0 \leq k \leq \dim(X)$, and consider the $k$-th intermediate Jacobian of $X_\CC$ \cite{voisin,mixedhodge}:
\begin{align} \label{intermediatejacobian}
\begin{split}
J^{2k-1}(X_\CC) &= \rm H^{2k-1}(X(\CC), \CC)/ \left( \Lambda + F^k\rm H^{2k-1}(X(\CC), \CC)\right), \\
\Lambda &= \rm H^{2k-1}(X(\CC), \ZZ(k))/\bracktorsion. 
\end{split}
\end{align}
Observe that $\sigma$ induces an anti-holomorphic involution $ J^{2k-1}(X_\CC) \to J^{2k-1}(X_\CC)$. Consider the Abel-Jacobi map  $
\rm{AJ}_\CC \colon \CH^k({X_\CC})_{\tn{hom}} \to J^{2k-1}(X_\CC)
$ introduced by Griffiths \cite{griffiths-certainrationalintegrals}. By \cite[\S7]{esnaultviehweg}, the Abel-Jacobi map fits naturally in the following commutative diagram:
\begin{align} \label{voisinsequence}
\xymatrix{
0 \ar[r] & \ca Z^k(X_\CC)_{\tn{hom}} \ar[r] \ar[d]^{\rm{AJ}_\CC} & \ca Z^k(X_\CC) \ar[r] \ar[d]^{cl_{\CC, D}} &  \ca Z^k(X_\CC) /  \ca Z^k(X_\CC)_{\tn{hom}} \ar[r]  \ar[d]^{cl_{\CC}} & 0\\
0 \ar[r]& J^{2k-1}(X_\CC) \ar[r]& \rm H^{2k}_D(X(\CC), \ZZ(k)) \ar[r]& \Hdg^{2k}(X(\CC), \ZZ(k)) \ar[r] & 0.
}
\end{align}
\begin{lemma} \label{lemma:GequivariantAbelJacobi}
With respect to the canonical $G$-actions, each arrow in (\ref{voisinsequence}) is $G$-equivariant. 
\end{lemma}
\begin{proof}
The $G$-equivariance of the horizontal homomorphisms in (\ref{voisinsequence}) is straightforward, and the fact that the vertical cycle class map on the right is $G$-equivariant is classical. To prove the lemma, it suffices to show that Deligne cycle class map $cl_{\CC, D}$ is $G$-equivariant. 
For this, consider a codimension $k$ subvariety $Y \subset X_\CC$. By \cite[p.172]{mixedhodge}, Deligne cohomology with supports in $Y(\CC)$ can be expressed as a fibre product 
\begin{align*}
\rm H^{2k}_{Y(\CC)}(X(\CC), \ZZ_D(k)) = \rm H^{2k}_{Y(\CC)}(X(\CC), \ZZ(k)) \times_{\rm H^{2k}_{Y(\CC)}(X(\CC), \CC)} \rm H^{2k}_{Y(\CC)}(X(\CC), F^k\Omega^\bullet_{X(\CC)}). 
\end{align*}
There is a unique class $\tau_{\tn{Del}}(Y) \in \rm H^{2k}_{Y(\CC)}(X(\CC), \ZZ_D(k))$ that lifts both the Thom class $\tau(Y)$ in $\rm H^{2k}_{Y(\CC)}(X(\CC), \ZZ(k))$ and the Thom-Hodge class $\tau_{\Hdg}(Y) \in \rm H^{2k}_{Y(\CC)}(X(\CC), F^k\Omega^\bullet_{X(\CC)})$, see \cite[Proposition 7.19]{mixedhodge}. 
The equality $\sigma^\ast(\tau(Y)) = \tau(Y^\sigma)$ is classical. It is also clear that 
$\sigma^\ast(\tau_{\Hdg}(Y)) = \tau_{\Hdg}(Y^\sigma)$; see the construction of $\tau_\Hdg(Y)$ in the proof of Theorem 2.38 in \cite{mixedhodge}. The lemma follows. 
\end{proof}

\noindent
By Lemma \ref{lemma:GequivariantAbelJacobi}, we obtain a \emph{real Abel-Jacobi map} 
\begin{align} \label{real-abel}
\tn{AJ} \colon \;  \CH^k(X)_{\tn{hom}} \to J^{2k-1}(X_\CC)^G. 
\end{align}
The boundary map in group cohomology induced by the exact sequence of $G$-modules
\begin{align} \label{groupcohomology}
0 \to \Lambda \to \rm H^{2k-1}(X(\CC),\RR(k)) \to J^{2k-1}(X_\CC) \to 0
\end{align}
induces an isomorphism between $\pi_0(J^{2k-1}(X_\CC)^G)$, the group of connected components of the real locus of $J^{2k-1}(X_\CC)$, and $\rm H^1(G, \Lambda)$. The spectral sequence (\ref{hochschild}) 
gives a filtration $F^\bullet$ on $\rm H^{2k}_G(X(\CC), \ZZ(k))$, together with natural maps $F^1\rm H^{2k}_G(X(\CC), \ZZ(k)) \to \rm H^1(G, \Lambda)$ 
and $\rm H^{2k}_G(X(\CC), \ZZ(k))  \to \rm H^{2k}(X(\CC), \ZZ(k))^G$, such that 
$cl \colon \CH^k(X)_{\tn{hom}} \to \rm H^{2k}_G(X(\CC), \ZZ(k))$ factors through $F^1\rm H^{2k}_G(X(\CC), \ZZ(k)) \subset \rm H^{2k}_G(X(\CC), \ZZ(k))$. 


\begin{theorem} \label{abeljacobicommutativity}
Let $X$ be a smooth projective variety over $\RR$. Let $k$ be an integer with $0 \leq k \leq \dim(X)$, and $\Lambda = \rm H^{2k-1}(X(\CC), \ZZ(k))/\bracktorsion$.  
The following diagram commutes:
\[
\xymatrixcolsep{3pc}
\xymatrix{
\CH^k(X)_{\tn{hom}} \ar[d]^{\tn{AJ}} \ar[r]^{cl\hspace{8mm}} & F^1\rm H^{2k}_G(X(\CC),\ZZ(k)) \ar[d]  \\
J^{2k-1}(X_\CC)^G \ar[r] & \rm H^1(G, \Lambda). 
}
\]
\end{theorem}
\noindent
To prove Theorem \ref{abeljacobicommutativity}, we need some notation. For a subring $R \subset \RR$ and an integer $d \in \ZZ_{\geq 0}$, define $R_D(d)$ as the complex 
\begin{align}\label{eq:gsequence}
0 \to R(d) \to \OO_{X(\CC)} \to \Omega_{X(\CC)}^1 \to \cdots \to \Omega_{X(\CC)}^{d-1}  \to 0. 
\end{align}
By definition, $\rm H^m_D(X(\CC), R(d))  \coloneqq \rm H^m(X(\CC), R_D(d))$ is the degree $m$ Deligne cohomology group with coefficients in $R(d)$ \cite{voisin, mixedhodge}. Note that $G$ acts on the complex (\ref{eq:gsequence}). 
\begin{definition}
For an integer $m \in \ZZ_{\geq 0}$, define the \emph{$G$-equivariant Deligne cohomology group of degree $m$ with coefficients in $R(d)$} as $
\rm H^m_{G,D}(X(\CC), R(d))  =
\rm H^m_G(X(\CC), R_D(d))$. Define the \emph{$k$-th equivariant intermediate Jacobian} as $J^{2k-1}_G(X) = \rm H^{2k-1}_G(X(\CC), \RR/\ZZ(k)).$
\end{definition}
\noindent
The exact sequence of $G$-complexes $
0 \to \Omega_{X(\CC)}^{\leq k-1}[-1] \to \ZZ_D(k) \to \ZZ(k) \to 0$ induces a natural homomorphism $
\rm H^{2k}_{G,D}(X(\CC), \ZZ(k)) \to \rm H^{2k}_{G}(X(\CC), \ZZ(k))$, and we have:

\begin{proposition} \label{prop:prop:prop}
There is a natural cycle class map $ \varphi \colon \ca Z^k(X) \to \rm H^{2k}_{G,D}(X(\CC),\ZZ(k))$ such that the following diagram commutes:
\[
\xymatrix{
\ca Z^k(X) \ar@/^2pc/[rr]^{cl} \ar[d] \ar[r]^{\varphi \hspace{9mm}} & \rm H^{2k}_{G,D}(X(\CC), \ZZ(k)) \ar[r] \ar[d] & \rm H^{2k}_G(X(\CC), \ZZ(k)) \ar[d] \\
\ca Z^k(X_\CC)\ar[r]\ar@/_2pc/[rr]_{cl_\CC}  & \rm H^{2k}_D(X(\CC), \ZZ(k)) \ar[r] & \rm H^{2k}(X(\CC), \ZZ(k)). 
}
\]
\end{proposition}
\begin{proof}
We adapt the proof of Proposition 7.19 in \cite{mixedhodge}. One has a canonical quasi-isomorphism 
$
R_D(k) \simeq \tn{Cone}^\bullet\left(
\epsilon^k - \iota^k \colon R(k) \oplus F^k\Omega_{X(\CC)}^\bullet \to \Omega_{X(\CC)}^\bullet
\right)[-1]$ for any ring $R \subset \RR$, see \cite[Lemma-Definition 7.13]{mixedhodge}. Thus, by \cite[(A-12)]{mixedhodge}, there is an exact sequence
\begin{align}\label{cone}
\Omega_{X(\CC)}^\bullet \to \ZZ_D(k)[1] \to \ZZ(k)[1] \oplus F^k\Omega_{X(\CC)}^\bullet[1]. 
\end{align}
The complexes and morphisms in (\ref{cone}) are all $G$-equivariant. Let $Y \subset X$ be a codimension $k$ subvariety over $\RR$. Taking $G$-equivariant cohomology with supports in $Y$ of the sequence (\ref{cone}) gives a long exact sequence 
\begin{align*}
\cdots \to & \rm H^{2k-1}_{G,Y(\CC)}(X(\CC), \Omega_{X(\CC)}^\bullet) \to  \rm H^{2k}_{G,Y(\CC)}(X(\CC), \ZZ_D(k))
  \\
 \to & \rm H^{2k}_{G,Y(\CC)}(X(\CC), \ZZ(k)) \oplus \rm H^{2k}_{G,Y(\CC)}(X(\CC), F^k\Omega_{X(\CC)}^\bullet)\to 
\rm H^{2k}_{G,Y(\CC)}(X(\CC), \Omega_{X(\CC)}^\bullet) \to \cdots 
\end{align*}
which is canonically isomorphic to the long exact sequence
\begin{align*}
\cdots \to & \rm H^{2k-1}_{Y(\CC)}(X(\CC), \CC)^G \to  \rm H^{2k}_{G,Y(\CC)}(X(\CC), \ZZ_D(k))
  \\
 \to & \rm H^{2k}_{G,Y(\CC)}(X(\CC), \ZZ(k)) \oplus \rm H^{2k}_{Y(\CC)}(X(\CC), F^k\Omega_{X(\CC)}^\bullet)^G\to 
\rm H^{2k}_{Y(\CC)}(X(\CC), \CC)^G \to \cdots. 
\end{align*}
The group $\rm H^{2k-1}_{Y(\CC)}(X(\CC), \CC)$ vanishes because by Poincar\'e-Lefschetz duality (see \cite[Section 19.1]{fultonintersection} or \cite[Theorem B.28]{mixedhodge}), one has
$
\rm H^{2k-1}_{Y(\CC)}(X(\CC), \CC) 
\cong \rm H^{\tn{BM}}_{2n-(2k-1)}(Y(\CC), \CC) = (0).
$
We conclude that $\rm H^{2k}_{G,Y(\CC)}(X(\CC), \ZZ_D(k))$ can be expressed as a fibre product
\[
\rm H^{2k}_{G,Y(\CC)}(X(\CC), \ZZ_D(k)) = \rm H^{2k}_{G,Y(\CC)}(X(\CC), \ZZ(k)) \times_{\rm H^{2k}_{Y(\CC)}(X(\CC), \CC)^G} \rm H^{2k}_{Y(\CC)}(X(\CC), F^k\Omega_{X(\CC)}^\bullet)^G. 
\]
Let $[Y] \in \rm H^{2k}_{G, Y(\CC)}(X(\CC), \ZZ(k))$ be the image of $1 \in \ZZ$ under the canonical isomorphism $\rm H^{2k}_{G, Y(\CC)}(X(\CC), \ZZ(k)) = \ZZ$, see \cite[equation (1.54)]{BW20}. 
Then $[Y]$ maps to the Thom class $\tau(Y)$ under 
$
\rm H^{2k}_{G, Y(\CC)}(X(\CC), \ZZ(k)) \to \rm H^{2k}_{G, Y(\CC)}(X(\CC), \CC) = \rm H^{2k}_{Y(\CC)}(X(\CC), \CC)^G,
$
so there is a unique class $[Y]_D \in \rm H^{2k}_{G,Y(\CC)}(X(\CC), \ZZ_D(k))$ that maps to $[Y]$ in the group $\rm H^{2k}_{G, Y(\CC)}(X(\CC), \ZZ(k))$ and to $\tau_{\tn{Hdg}}(Y)$ in the group $\rm H^{2k}_{Y(\CC)}(X(\CC), F^k\Omega^\bullet_{X(\CC)})$. Upon forgetting the support, we obtain an equivariant Deligne cycle class $[Y]_D \in \rm H^{2k}_{G,D}(X(\CC), \ZZ(k)$. 
\end{proof}
\noindent
The natural isomorphism $J^{2k-1}(X_\CC) = \rm H^{2k-1}(X(\CC), \RR/\ZZ(k))$ of $G$-modules gives a homomorphism $J^{2k-1}_G(X) \to J^{2k-1}(X_\CC)^G$. 
Consider the short exact sequence of $G$-complexes
\begin{align}\label{gsh}
0 \to \ZZ_D(k) \to \RR_D(k) \to \RR/\ZZ \to 0.
\end{align}

\begin{lemma} \label{lemma:bigdiagram}
Sequence (\ref{gsh}) induces the following commutative diagram with exact rows:
\begin{align} \label{bigdiagram}
\scriptsize{
\xymatrixcolsep{0.0005pc}
\xymatrix{
 0 \ar[r] &
 J_G^{2k-1}(X) \ar[rr]\ar[dl] \ar[dd]&&
\rm H^{2k}_{G,D}(X(\CC),\ZZ(k)) \ar[dl]\ar[rr]\ar[dd] &&
\rm H^{2k}_{G,D}(X(\CC), \RR(k))\ar[dd]^{\wr} \\
 F^1\rm H^{2k}_G(X(\CC),\ZZ(k)) \ar[rr] \ar@/_1pc/[dd] &&\Hdg_G^{2k}(X(\CC), \ZZ(k)) \ar[rr]&&\Hdg^{2k}(X(\CC), \ZZ(k))^G\ar@{_{(}->}[dr]& \\
0 \ar[r] & J^{2k-1}(X_\CC)^G\ar[dl]  \ar[rr] && 
\rm H^{2k}_D(X(\CC), \ZZ(k))^G\ar[ur] \ar[rr] &&
\rm H^{2k}_D(X(\CC), \RR(k))^G.  \\
\rm H^1(G, \Lambda) &&&&&
}}
\end{align}
\normalsize
\end{lemma}
\begin{proof}
The morphism
$
\rm H^{2k}_D(X(\CC), \ZZ(k)) \to \rm H^{2k}_D(X(\CC), \RR(k))
$
factors through an embedding 
$
\Hdg^{2k}(X(\CC), \ZZ(k)) \hookrightarrow \rm H^{2k}_D(X(\CC), \RR(k))$, which gives the triangle on the bottom right of (\ref{bigdiagram}). Moreover, the vertical arrow on the right of (\ref{bigdiagram}) is an isomorphism because $\RR_D(k)$ is a $G$-complex of $\RR$-vector spaces. 

Next, we show that the boundary map $ J_G^{2k-1}(X) \to \rm H^{2k}_{G,D}(X(\CC), \ZZ(k))$ induced by (\ref{gsh}) is injective. This follows from the fact that the canonical homomorphism 
\[
f \colon \rm H^{2k-1}_{G,D}(X(\CC), \RR(k)) \to J^{2k-1}_G(X)
\]
is zero. To prove this, note that $f$ factors as 
\[
{\small
\xymatrix{
\rm H^{2k-1}_{G,D}(X(\CC), \RR(k))\ar@/^2pc/[rrr]^{f}  \ar[r] \ar[d]^{\wr} & \rm H^{2k-1}_G(X(\CC), \RR(k))\ar[d]^{\wr}  \ar[r] & \rm H^{2k-1}_G(X(\CC), \RR/\ZZ(k))  \ar@{=}[r] & J^{2k-1}_G(X) \\
\rm H^{2k-1}_D(X(\CC), \RR(k))^G \ar[r]^g & \rm H^{2k-1}(X(\CC), \RR(k))^G, & &
}
}
\]
so that it suffices to show that the map $g \colon \rm H^{2k-1}_D(X(\CC), \RR(k)) \to \rm H^{2k-1}(X(\CC), \RR(k))$ is zero. This holds, because the exact sequence $
0 \to \Omega_{X(\CC)}^{\leq k-1}[-1] \to \RR_D(k) \to \RR(k)  \to 0
$
induces an exact sequence
\[
\cdots \to \rm H^{2k-1}_D(X(\CC), \RR(k)) \xrightarrow{g} 
\rm H^{2k-1}(X(\CC), \RR(k)) \to \bb H^{2k-1}(X(\CC), \Omega_{X(\CC)}^{\leq k-1}) \to \cdots 
\]
in which
\begin{align}\label{voisiniso}
\rm H^{2k-1}(X(\CC), \RR(k)) = \rm H^{2k-1}(X(\CC), \CC)/F^k\rm H^{2k-1}(X(\CC), \CC)
=
\bb H^{2k-1}(X(\CC), \Omega_{X(\CC)}^{\leq k-1}). 
\end{align}
See the proof of Proposition 12.26 in \cite{voisin} for the second isomorphism in (\ref{voisiniso}). 

Finally, the map $J^{2k-1}_G(X) \to \rm H^{2k}_{G,D}(X(\CC), \ZZ(k))$ is the boundary map in cohomology induced by an exact sequence of $G$-complexes, thus shifts the Hochschild-Serre filtrations on both sides by one degree. This gives a diagram in which both squares commute:
\[
\xymatrix{
J^{2k-1}_G(X) \ar[r] \ar[d] & F^1\rm H^{2k}_{G,D}(X(\CC), \ZZ(k)) \ar[r]\ar[d] & F^1\rm H^{2k}_G(X(\CC), \ZZ(k)) \ar[d]\\
J^{2k-1}(X_\CC)^G \ar[r] & \rm H^1(G, \rm H^{2k-1}(X(\CC), \ZZ_D(k))) \ar[r] & \rm H^1(G, \Lambda).  
}
\]
Consequently, the outer square commutes as well, and we are done. 
\end{proof}
\noindent
By Proposition \ref{prop:prop:prop} and Lemma \ref{lemma:bigdiagram}, we obtain a commutative diagram with exact rows:
\begin{align}\label{double}
\begin{split}
\xymatrixrowsep{1pc}
\xymatrix{
0 \ar[r]& \ca Z^k(X)_{\tn{hom}}\ar@{.>}[d] \ar[r] &
\ca Z^k(X)\ar[d] \ar[r]& \ca Z^k(X)/\ca Z^k(X)_{\tn{hom}} \ar[r] \ar[d]& 0 \\
0 \ar[r]& J^{2k-1}_G(X) \ar[d] \ar[r]& \rm H^{2k}_{G,D}(X(\CC), \ZZ(k)) \ar[r] \ar[d] & \rm H^{2k}(X(\CC), \ZZ(k))^G \ar@{=}[d] \\ 
0 \ar[r] & F^1\rm H^{2k}_G(X(\CC), \ZZ(k)) \ar[r] &\rm H^{2k}_G(X(\CC), \ZZ(k)) \ar[r] & \rm H^{2k}(X(\CC), \ZZ(k))^G \ar[r] & 0. 
}
\end{split}
\end{align}
\begin{lemma} \label{griffithscommutativity}
The composition of the induced morphism $\ca Z^k(X)_{\tn{hom}} \to J_G^{2k-1}(X)$ with $J^{2k-1}_G(X) \to J^{2k-1}(X_\CC)$ coincides with $\ca Z^k(X)_{\tn{hom}} \to \ca Z^k(X_\CC)_{\tn{hom}} \to J^{2k-1}(X_\CC)$.
\end{lemma}
\begin{proof}
 Consider the following diagram:
 \[
 \xymatrixrowsep{0.8pc}
 \xymatrix{
 &\ca Z^k(X)_{\tn{hom}}\ar[dl] \ar[dd] \ar@{^{(}->}[rr] && \ca Z^k(X) \ar[dd]\ar[dl] \\
\ca Z^k(X_\CC)_{\tn{hom}}\ar@{^{(}->}[rr]\ar[dd]^{\tn{AJ}_\CC} &&\ca Z^k(X_\CC)\ar[dd]&\\
 &J^{2k-1}_G(X) \ar[dl]\ar@{^{(}->}[rr] && \rm H^{2k}_{G,D}(X(\CC), \ZZ(k)) \ar[dl] \\
 J^{2k-1}(X_\CC) \ar@{^{(}->}[rr] && \rm H^{2k}_D(X(\CC), \ZZ(k)). &
 }
 \]
 By Proposition \ref{prop:prop:prop} and Lemma \ref{lemma:bigdiagram}, each square in this diagram, except possibly the square on the left, commutes. By injectivity of the two horizontal arrows in the bottom square, the square on the left commutes as well. 
\end{proof}

\begin{proof}[Proof of Theorem \ref{abeljacobicommutativity}]
The vertical composition on the left of the diagram (\ref{double}) coincides with the cycle class map $cl \colon \ca Z^k(X)_{\tn{hom}} \to F^1\rm H^{2k}_G(X(\CC), \ZZ(k))$ because the analogous statement for the vertical composition in the middle of (\ref{double}) holds, see Proposition \ref{prop:prop:prop}. We are done because by Lemma's \ref{lemma:bigdiagram} and \ref{griffithscommutativity}, each square in the following diagram commutes: 
\[
\xymatrix{
\ca Z^k(X)_{\tn{hom}} \ar[d] \ar[r] 
& J^{2k-1}_G(X) \ar[r] \ar[d] & F^1\rm H^{2k}_G(X(\CC), \ZZ(k)) \ar[d] \\
\ca Z^k(X_\CC)^G_{\tn{hom}} \ar[r] & J^{2k-1}(X_\CC)^G \ar[r] & \rm{H}^1(G, \Lambda).
}
\]
\end{proof}

\section{Codimension-two cycles on real abelian varieties} \label{sec:codim-two}

We begin with the proof of Theorem \ref{fourierreduction}. The proof consists of two steps: first we prove the theorem for the Jacobian of a real algebraic curve with non-empty real locus, and then we reduce the general case to this particular case. 

\begin{proof}[Proof of Theorem \ref{fourierreduction} (Jacobian case)] Let us prove Theorem \ref{fourierreduction} in the case where $A$ is the Jacobian $J = J(C)$ of a real algebraic curve $C$ of genus $g \in \ZZ_{\geq 1}$ such that $C(\RR) \neq \emptyset$. By Proposition \ref{vanishing}, 
one needs to show that $F^2\rm H^4_G(J(\CC),\ZZ(2))_0$ is algebraic. 
By Theorem \ref{grabowski}, 
the Chern character of the Poincar\'e bundle of $J$ in cohomology $\rm H^{2\bullet}(J(\CC) \times \wh J(\CC), \ZZ(\bullet))^G$ is algebraic. 
By Section \ref{subsec:analysis} and \cite[Corollary 9.24]{huybrechtsfouriermukai}, 
we obtain a commutative diagram
{\small 
\begin{align}\label{bottom}
\xymatrix{
F^2\rm H^4_G(J(\CC),\ZZ(2))_0\ar@{=}[d]  \ar@{^{(}->}[r]^{\Gamma_\ast^{4,2g}} &  F^2\rm H^{2g}_G(\wh J(\CC),\ZZ(g))_0 \ar[d] \ar@{->>}[r]^{\Gamma_\ast^{2g,4}} & F^2\rm H^4_G(J(\CC),\ZZ(2))_0\ar@{=}[d]  \\
\rm H^2(G, \rm H^2(J(\CC),\ZZ(2)))_0 \ar[r]^{\sim\hspace{1em}} &  \rm H^2(G, \rm H^{2g-2}(\wh J(\CC),\ZZ(g)))_0 \ar[r]^{\sim} & \rm H^2(G, \rm H^2(J(\CC),\ZZ(2)))_0
}
\end{align}
}
such that the composition on the bottom row of (\ref{bottom}) is the identity. Therefore, the composition on the top row of (\ref{bottom}) is the identity. Since $J$ satisfies the real integral Hodge conjecture for zero-cycles by \cite[Proposition 2.10]{BW20}, the group $F^2\rm H^{2g}_G(J(\CC),\ZZ(g))_0$ -- and hence also the group $F^2\rm H^4_G(J(\CC),\ZZ(2))_0$ -- is algebraic. 
\end{proof}
\noindent
The following notation will be used throughout the remaining part of Section \ref{sec:codim-two}. 
\begin{definition}
For a smooth projective variety $X$ of dimension $n$ over $\RR$, let $\CH_0(X)_0$ be the group of zero-cycles of degree zero. Observe that $\CH_0(X)_0 = \CH_0(X)_{\tn{hom}}$. 
\end{definition}

\begin{lemma} \label{anotherlemma} Let $f \colon J \to A$ be a surjective homomorphism of real abelian varieties. 
\begin{enumerate}
\item \label{anotherlemma:first}
The induced homomorphism $f_\RR^0 \colon J(\RR)^0 \to A(\RR)^0$ is surjective. 
\item 
If $f_\RR \colon J(\RR) \to A(\RR)$ is surjective, then $\phi = f_\ast \colon \CH_0(J)_0 \to \CH_0(A)_0$ is surjective. 
\end{enumerate}
\end{lemma}
\begin{proof}
(i). 
Let $\rm{Nm}_X \colon X(\CC) \to X(\RR)^0, x \mapsto x + \sigma(x)$ be the norm homomorphism of a real abelian variety $X$. One has that $f_\RR^0 \circ \rm{Nm}_J = \rm{Nm}_A \circ f_\CC$. By \cite[Proposition 1.1]{grossharris}, the homomorphism $\rm{Nm}_A$ is surjective. Thus $f_\RR^0$ is surjective.

(ii). 
Indeed, for an abelian variety $X$ over a field $k$, the group $\CH_0(X)_0$ is generated by zero-cycles of the form $[x] - \deg(k(x_i)/k)\cdot [0]$ for closed points $x$ on $X$.
\end{proof}

\begin{proof}[Proof of Theorem \ref{fourierreduction} (general case)]
To prove that $F^2\rm H^{4}_G(A(\CC),\ZZ(2))_0$ is algebraic for a real abelian variety $A$, we would like to reduce to the Jacobian case. The fact that this can be done rests on the following proposition in combination with Lemma \ref{gabber} below. 
\begin{proposition} \label{keyprop}
Let $0 \to B \to J \xrightarrow{f} A \to 0$ be an exact sequence of abelian varieties over $\RR$ such that $f_\RR \colon J(\RR) \to A(\RR)$ is surjective. Let $d = \dim(J)$ and $g = \dim(A)$. Then 
\begin{align*}
& f_\ast \colon F^2\rm H^{2d}_G(J(C)(\CC),\ZZ(d))_0 \to F^2\rm H^{2g}_G(A(\CC),\ZZ(g))_0 \quad\quad \text{ and } \\
 & \hat{f}^\ast \colon F^2\rm H^{4}_G(\wh J(C)(\CC),\ZZ(2))_0 \to F^2\rm H^{4}_G(\wh A(\CC),\ZZ(2))_0
 \end{align*}
are surjective, where $\wh f \colon \wh A \to \wh J$ is the dual homomorphism of $f \colon J \to A$. 
\end{proposition}
\begin{proof}[Proof of Proposition \ref{keyprop}]
We will need some notation. 
\begin{definition}  \label{F2def}
For an abelian variety $A$ of dimension $g$ over $\RR$, define
\begin{align*}
F^2\CH_0(A)_0 &= \Ker \left(  \CH_0(A)_0 \to \rm H^1(G, \rm H^{2g-1}(A(\CC), \ZZ(g)))\right) \\
&= \left\{ \alpha \in \CH_0(A)_0 \mid \textnormal{AJ}(\alpha) \in A(\RR)^0 \right\}.
\end{align*}
\end{definition}
\noindent
For the second equality in Definition \ref{F2def}, see Theorem \ref{abeljacobicommutativity}.  
For an abelian variety $X$ of dimension $n$ over $\RR$, define $\Lambda_X = \rm H^{2n-1}(X(\CC), \ZZ(n))$. Poincar\'e duality identifies $\Lambda_X$ with $\rm H_1(X(\CC), \ZZ)$ \cite[Corollaire 3.1.9]{mangolte}, and we have $\pi_0(X(\RR)) = \rm H^1(G, \Lambda_X)$ (see (\ref{groupcohomology})). 
\begin{enumerate}[leftmargin=-0.08cm, rightmargin = -0.08cm]
    \item[Step 1:] \emph{The kernel of the push-forward $\phi \colon \CH_0(J)_{0} \to \CH_0(A)_{0}$ surjects onto the kernel of the homomorphism $f \colon \pi_0(J(\RR)) \to \pi_0(A(\RR))$.} It suffices to show that the composition
    \begin{align}\label{eaaa}
 \CH_0(B)_0 \to \Ker(\phi) \to \Ker(f)
    \end{align}
    is surjective. Now (\ref{eaaa}) coincides with the composition $\CH_0(B)_0 \to \pi_0(B(\RR)) \to \Ker(f)$ of surjections; the map $\pi_0(B(\RR)) \to \Ker(f)$ is surjective by 
the exact sequence 
\begin{align*} 
0 \to \Lambda_B^G \to \Lambda_J^G \to \Lambda_A^G \to \rm H^1(G, \Lambda_B) \to \rm H^1(G, \Lambda_J) \to \rm H^1(G, \Lambda_A) \to 0 
\end{align*}
arising from the short exact sequence of $G$-modules $0 \to \Lambda_B \to \Lambda_J \to \Lambda_A \to 0$. 
    \item[Step 2:] \emph{Surjectivity of $f_\ast$}: Consider the following commutative diagram:
\[
\xymatrixrowsep{1.3pc}
\xymatrix{
F^2\CH_0(J)_{0} \ar[r]^{\phi|_{F^2}} \ar[d] & F^2\CH_0(A)_{0} \ar[d] \\
F^2\rm H^{2d}_G(J(\CC), \ZZ(d))_0 \ar[r]^{f_\ast} & F^2\rm H^{2g}_G(A(\CC), \ZZ(g))_0. 
}
\]
Its vertical arrows are surjective by the real integral Hodge conjecture for zero-cycles and Definition \ref{F2def}. To prove the surjectivity of $f_\ast$ on the bottom row, it thus suffices to prove the surjectivity of $\phi|_{F^2}$ on the top row. 
By Definition \ref{F2def}, the rows in the following commutative diagram are exact:
\begin{align*}
\xymatrixrowsep{1.3pc}
\xymatrix{
0 \ar[r] & F^2\CH_0(J)_{0} \ar[r] \ar[d]^{\phi|_{F^2}} & \CH_0(J)_{0} \ar[r] \ar[d]^\phi& \rm H^1(G, \Lambda_J) \ar[r] \ar[d]^f & 0  \\
0 \ar[r] & F^2\CH_0(A)_{0} \ar[r] & \CH_0(A)_{0} \ar[r]& \rm H^1(G, \Lambda_A) \ar[r] & 0.  
}
\end{align*}
Note that $\phi$ is surjective by Lemma \ref{anotherlemma}. Since $J(\RR) \to A(\RR)$ is surjective we have that the map $f \colon \pi_0(J(\RR)) \to \pi_0(A(\RR))$ is surjective. By the snake lemma, to prove that $\phi|_{F^2}$ is surjective, it suffices to show that $\Ker(\phi) \to \Ker(f)$ is surjective, which is Step 1. 
\item[Step 3:]\emph{The maps $f_\ast$ and $\hat{f}^\ast$ fit into the following commutative diagram, where the arrows $\twoheadrightarrow$ are surjective and the arrows $\xrightarrow{\sim}$ are isomorphisms}:
{\small
\begin{align} \label{hugediagram}
\xymatrixrowsep{1.2pc}
\xymatrixcolsep{0.3pc}
\xymatrix{
F^2\rm H^{2d}_G(J(\CC), \ZZ(d))_0 \ar@{->>}[dd]^{f_\ast}\ar@{->>}[dr]& & F^2\rm H^4_G(\wh J(\CC), \ZZ(2))_0\ar[dd]^(.35){\hat{f}^\ast}
\ar[dr]^{\sim}& \\
& \rm H^2(G, \rm H^{2d-2}(J(\CC), \ZZ(d)))_0\ar[dd]^{\rm H^2(G, f_\ast)}\ar[rr]_{\rm H^2(G, \mr F_J)\quad\quad\quad}^{\sim\quad\quad\quad\quad} &&\rm H^2(G, \rm H^2(\wh J(\CC), \ZZ(2)))_0\ar[dd]^{\rm H^2(G, \hat{f}^\ast)} \\
F^2\rm H^{2g}_G(A(\CC), \ZZ(g))_0 \ar@{->>}[dr]& & F^2\rm H^4(\wh A(\CC), \ZZ(2))_0 \ar[dr]^{\sim}& \\
& \rm H^2(G, \rm H^{2g-2}(A(\CC), \ZZ(g)))_0 \ar[rr]^\sim_{\rm H^2(G, \mr F_A)}&&\rm H^2(G, \rm H^2(\wh A(\CC), \ZZ(2)))_0. 
}
\end{align}
}
Indeed, this follows from the functoriality of Fourier transforms (see \cite[(3.7.1)]{moonenpolishchuk}), together with Proposition \ref{vanishing} and Lemma \ref{topologicalcompatibility}. 
\item[Step 4:] \emph{Surjectivity of $\hat{f}^\ast$}: The surjectivity of $f_\ast$ on the left of (\ref{hugediagram}) was shown in Step 2. By the commutativity on the left of diagram (\ref{hugediagram}), the morphism 
$
\rm H^2(G, f_\ast)$ 
is surjective. By the commutativity of the square on the front, this implies that $\rm H^2(G, \hat{f}^\ast)$ 
is surjective which, by the commutativity on the right hand side of the diagram, implies that the morphism $$\hat{f}^\ast \colon F^2\rm H^4_G(\wh J(\CC), \ZZ(2))_0 \to F^2\rm H^4(\wh A(\CC), \ZZ(2))_0$$ is surjective. 
  \end{enumerate}
 This finishes the proof of Proposition \ref{keyprop}. 
\end{proof}
\begin{lemma} \label{gabber}
Let $A$ be an abelian variety over $\RR$. Then $A$ contains a smooth, proper, geometrically connected curve $C$ over $\RR$ that passes through $0 \in A(\RR)$ in such a way that the following two conditions hold: the induced homomorphism $f \colon J(C) \to A$ is surjective with connected kernel, and the homomorphism $f_\RR \colon J(C)(\RR) \to A(\RR)$ is surjective. 
\end{lemma}
\begin{proof}
Let $S \subset A(\RR)$ be a finite set of points containing $0 \in A(\RR)$ and at least one point of each connected component of $A(\RR)$. Since $\RR$ is infinite, Bertini's theorem can be applied: there exists a smooth and geometrically connected hyperplane section $Z \subset A$ passing through $S$ \cite[II, Theorem 8.18]{HAG}, \cite[Footnote 12, page 32]{debarrebook}. Let $g = \dim(A)$. By taking $g-1$ general such hyperplane sections, we get a smooth, geometrically connected curve $C$ in $A$ that contains $S$. We claim that $C$ satisfies the requirements stated in the proposition. Write $J = J(C)$ and consider the map $f \colon J \to A$ arising from the inclusion $(C,0) \hookrightarrow (A,0)$. By the proof of Theorem 10.1 in \cite{milnejacobian}, the homomorphism $f$ is surjective. Moreover, the kernel of $f \colon J_\CC \to A_\CC$ is connected. Indeed, this follows from the fact that $\rm H_1(C(\CC), \ZZ) \to \rm H_1(A(\CC), \ZZ)$ is surjective by the Lefschetz hyperplane theorem; alternatively, see \cite[Proposition 2.4]{gabberspacefilling} for an algebraic argument. The morphism $f_\RR^0 \colon J(\RR)^0 \to A(\RR)^0$ is surjective by Lemma \ref{anotherlemma}. Since $S$ is contained in the image of $f_\RR$, we conclude that $f_\RR$ is surjective and we are done.  
\end{proof}
\noindent
We can now finish the proof of Theorem \ref{fourierreduction}. Let $A$ be an abelian variety over $\RR$. 
By Proposition \ref{vanishing}, it remains to prove that $F^2\rm H^4_G(A(\CC),\ZZ(2))_0$ is algebraic. Let $C \subset \wh A$ be a real algebraic curve that satisfies the conditions of Lemma \ref{gabber}. By Proposition \ref{keyprop}, the pull-back $\hat{f}^\ast \colon F^2\rm H^{4}_G(\wh J(C)(\CC),\ZZ(2))_0 \to F^2\rm H^{4}_G(A(\CC),\ZZ(2))_0$ is surjective, where $f \colon J(C) \to \wh A$ is the homomorphism induced by the inclusion of $(C,0)$ in $(\wh A, 0)$. By the proof of Theorem \ref{fourierreduction} in the Jacobian case, we know that $ F^2\rm H^{4}_G(\wh J(C)(\CC),\ZZ(2))_0$ is algebraic. Therefore $F^2\rm H^{4}_G(A(\CC),\ZZ(2))_0$ is algebraic. 
\end{proof}

\begin{proof}[Proof of Corollary \ref{questionreductioncorollary}]
Consider the filtration (\ref{topologically-distinguished-filtration}). By Lemma \ref{lemma:important0}, $A$ satisfies the real integral Hodge conjecture if and only if $F^1\rm H^4_G(A(\CC), \ZZ(2))_0$ and $\Hdg^4(A(\CC), \ZZ(2))^G$ are algebraic. Thus, the corollary follows from Theorems \ref{theorem1}, 
\ref{fourierreduction} and \ref{abeljacobicommutativity}. 
\end{proof}

\section{One-cycles on real abelian threefolds that split as a product} \label{sec:last}




\subsection{An elliptic curve with connected real locus times an abelian surface} 

\begin{proposition} \label{prop:productsurfacecurve}
Let $B$ be a real abelian surface, and let $E$ be a real elliptic curve whose real locus $E(\RR)$ is connected. The Abel-Jacobi homomorphism 
\begin{align}\label{surjectivityforproducts}
\rm{AJ}\; \colon \CH^2(B \times E)_{\hom} \to \rm H^1(G, \rm H^3(B(\CC) \times E(\CC), \ZZ(2))).
\end{align}
is surjective. In particular, Theorem \ref{th:abeliansurface} holds. 
\end{proposition}

\begin{proof}
We have  
$
\rm H^1(G, \rm H^2(B(\CC), \ZZ) \otimes \rm H^1(E(\CC), \ZZ)) = (0)
$
because $\rm H^1(E(\CC), \ZZ) \cong \ZZ[G]$; the latter follows from the fact that the real locus $E(\RR)$ of $E$ is connected \cite[Chapter IV, Example 3.1]{vanhamel}. By the K\"unneth formula, the canonical morphism 
\begin{align} \label{equiv:kunneth}
\begin{split}
\rm H^1(G, \rm H^3(B(\CC), \ZZ(2)) \;\; &\bigoplus \;\; \rm H^1(G, \rm H^1(B(\CC), \ZZ(1)) \otimes \rm H^2(E(\CC), \ZZ(1))) \\
  & \to \;\; \rm H^1(G, \rm H^3(B(\CC) \times E(\CC), \ZZ(2)))
\end{split}
\end{align}
is therefore an isomorphism. Since $\rm H^2(E(\CC), \ZZ(1)) \cong \ZZ$ as $G$-modules, the canonical map
\begin{align} \label{equiv:gmodstructure}
\begin{split}
\rm H^1(G, \rm H^1(B(\CC), \ZZ(1))) \;\; &\bigotimes \;\;  \rm H^0(G, \rm H^2(E(\CC), \ZZ(1))) \\
  & \to \;\;  \rm H^1(G, \rm H^1(B(\CC), \ZZ(1)) \otimes \rm H^2(E(\CC), \ZZ(1)))
  \end{split}
\end{align}
is an isomorphism as well. To simplify notation, for any abelian variety $X$ over $\RR$, define 
\begin{align}\label{lastequation}
\rm H^i_X(j) = \rm H^i(X(\CC), \ZZ(j)).
\end{align} 
Consider the following commutative diagram: 
{\small
\begin{align*} \label{equiv:important}
\begin{split}
\xymatrixrowsep{1pc}
\xymatrix{
\CH^2(B)_{0} \bigoplus  \left(\CH^1(B)_{\hom} \otimes \CH^1(E) \right) \ar[r] \ar[d] & \CH^2(B \times E)_{\hom} \ar[d] \\
F^1\rm H^4_G(B(\CC), \ZZ(2)) \bigoplus  \left(F^1\rm H^2_G(B(\CC), \ZZ(1)) \otimes \rm H^2_G(E(\CC), \ZZ(1)) \right) \ar[r]  \ar[d]& F^1\rm H^4_G(B(\CC) \times E(\CC), \ZZ(2)) \ar[dd] \\
\rm H^1(G, \rm H^3_B(2)) \bigoplus \left( \rm H^1(G, \rm H^1_B(1)) \otimes \rm H^0(G, \rm H^2_E(1)) \right) \ar[d]^{\wr} &  \\
\rm H^1(G, \rm H^3_B(2)) \bigoplus \rm H^1(G, \rm H^1_B(1) \otimes \rm H^2_E(1)) \ar[r]^{\hspace{10mm} \sim} & \rm H^1(G, \rm H^3_{B \times E}(2)). 
}
\end{split}
\end{align*}
}
The indicated isomorphisms $\xrightarrow{\sim}$ in this diagram arise from (\ref{equiv:kunneth}) and (\ref{equiv:gmodstructure}) above. The map $\CH^2(B)_{0} \to \rm H^1(G, \rm H^3_B(2))$ is surjective because by Theorem \ref{abeljacobicommutativity} it factors as 
\[
\CH^2(B)_0 = \CH_0(B)_0 \twoheadrightarrow B(\RR) \twoheadrightarrow \rm H^1(G, \rm H^3(B(\CC), \ZZ(2))). 
\]
The map $\CH^1(B)_{\hom} \to \rm H^1(G, \rm H^1_B(1))$ is surjective by the real integral Hodge conjecture for divisors (see Section \ref{subsec:divisors}) and the fact that the topological condition for degree two cohomology classes is trivial \cite[\S2.3.1]{BW20}. 
Thus, (\ref{surjectivityforproducts}) is surjective. By Corollary \ref{questionreductioncorollary}, the abelian variety $A = B \times E$ satisfies the real integral Hodge conjecture. 
\end{proof}
\noindent
The topological condition (\ref{topologicalcondition}) does not appear on the right hand side of (\ref{surjectivityforproducts}). 
This has the following corollary. 
For a real abelian variety $A$, one has $\va{\pi_0(A(\RR))} \leq 2^{\dim(A)}$ (see \cite[\S1]{grossharris}). Moreover, for every non-negative pair of integers $(i,g)$ with $i \leq g$, there is a real abelian variety $A$ of dimension $g$ such that $\va{\pi_0(A(\RR))} = 2^i$ (take a suitable product of elliptic curves). In particular, the map $\ca A_3(\RR) \to \ZZ$, $[A] \mapsto \va{\pi_0(A(\RR))}$ induces a bijection $\pi_0(\ca A_3(\RR)) \cong \{1,2,4,8\}$. 

\begin{corollary} \label{corollary:n(A)}
Let $A$ be a real abelian threefold, and suppose that $\va{\pi_0(A(\RR))} \neq 8$. The canonical map 
\begin{align} \label{AJsurjectivity}
F^1\rm H^4_G(A(\CC), \ZZ(2))_0 \to \rm H^1(G, \rm H^3(A(\CC), \ZZ(2)))
\end{align}
is surjective. Thus, $A$ satisfies the real integral Hodge conjecture if and only if the real Abel-Jacobi map $ \CH_1(A)_{\hom} \to \rm H^1(G, \rm H^3(A(\CC), \ZZ(2)))$ is surjective. 
\end{corollary}
\begin{proof}
To prove the surjectivity of (\ref{AJsurjectivity}), we may replace $A(\CC)$ by any differentiable $G$-manifold which is $G$-equivariantly diffeomorphic to $A(\CC)$. In particular, we may assume that 
$A = B \times E$, where $B$ is an abelian surface and $E$ an elliptic curve whose real locus is connected (see (\ref{topologicalstatement})). The surjectivity of (\ref{AJsurjectivity}) follows then from the fact that the group $\rm H^1(G, \rm H^3(B(\CC) \times E(\CC), \ZZ(2)))$ is algebraic by Proposition \ref{prop:productsurfacecurve}. 
\end{proof}

\subsection{The product of three real elliptic curves} \label{subsec:threeellipticcurves}Recall that $\va{\pi_0(A(\RR))} \in \{1,2,4,8\}$ for an abelian threefold $A$ over $\RR$. 
It follows from Theorem \ref{th:abeliansurface} that for each $m \in \{1,2,4\}$, there is a real abelian threefold $A$ with $\va{\pi_0(A(\RR))} = m$ such that $A$ satisfies the real integral Hodge conjecture. 
\begin{question} \label{question:pi8}
Does there exist an abelian threefold $A$ over $\RR$ with $\va{\pi_0(A(\RR))} = 8$ such that $A$ satisfies the real integral Hodge conjecture?
\end{question}
\noindent
In this section, we consider this question for the third self-product of a real elliptic curve. 

Let $E$ be an elliptic curve over $\RR$ with disconnected real locus. Choose a symplectic basis $\{x, y\}$ of $ \rm H^1(E(\CC), \ZZ)$ with $\sigma^\ast(x) = x$ and $\sigma^\ast(y) = -y$ (see e.g.\ \cite[\S 9]{grossharris}). Let $A = E^3$, and define $\pi_i \colon A(\CC) \to E(\CC)$ as the projection onto the $i$-th factor. We consider the basis element $\gamma \coloneqq \pi_1^\ast(x) \cdot \pi_2^\ast(y) \cdot \pi_3^\ast(x) \in \rm H^3(A(\CC), \ZZ)$, as well as the Galois cohomology class  
$
[\gamma] = [\pi_1^\ast(x) \cdot \pi_2^\ast(y) \cdot \pi_3^\ast(x)] \in \rm H^1(G, \rm H^3(A(\CC), \ZZ))
$ 
that $\gamma$ defines. 
Let $\beta \in F^1\rm H^2_G(E(\CC), \ZZ)$ be any element that lifts $[y]$. Let $\alpha \in \rm H^1_G(E(\CC), \ZZ)$ be the unique element that maps to $x \in \rm H^1(E(\CC), \ZZ)^G$. Finally, consider the canonical map 
\begin{align*}
c\; \colon \rm H^1_G(E(\CC), \ZZ) \otimes F^1\rm H^2_G(E(\CC), \ZZ) \otimes \rm H^1_G(E(\CC), \ZZ)  \to F^1\rm H_G^4(A(\CC), \ZZ). 
\end{align*}

\begin{proposition}  \label{prop:finalprop}
The element $\omega = c(\alpha \otimes \beta \otimes \alpha) \in F^1\rm H_G^4(A(\CC), \ZZ(2))$ satisfies the topological condition (\ref{topologicalcondition}), and every class in $F^1\rm H^4_G(A(\CC), \ZZ(2))_0$ is a $\ZZ/2$-linear combination of $\omega$ and algebraic elements. 
In particular, $A$ satisfies the real integral Hodge conjecture if and only if $\omega$ is algebraic if and only if $[\gamma] \in \rm H^1(G, \rm H^3(A(\CC), \ZZ(2)))$ is algebraic.  
\end{proposition}
\begin{lemma} \label{second-countable}
Let $X$ be a paracompact Hausdorff topological $G$-space and $F$ a $G$-sheaf on $X$. 
For $m > \dim(X)$, the canonical map $\rm H^m_G(X, F) \to \rm H^m_G(X^G, F|_{X^G})$ is an isomorphism. 
\end{lemma}
\begin{proof}
Define $Z = X^G$ and let $j \colon U = X \setminus Z \to X$ and $\iota \colon Z \to X$ be the canonical inclusions. Write $F_U = j_!(F|_U)$ and $F_Z = \iota_\ast(F|_Z)$. 
The canonical homomorphism $\rm H^i_G(Z, F|_Z) \to \rm H^i_G(X, F_Z)$ is an isomorphism, see \cite[Propositions 2.2.1 \& 3.1.4]{tohoku}, 
which gives an exact sequence 
$
\rm H^m_G(X, F_U) \to \rm H^m_G(X,F) \to \rm H^m_G(Z, F|_Z) \to \rm H^{m+1}_G(X, F_U). 
$
By \cite[Theorems 5.2.1 \& 5.3.1]{tohoku}, 
we have $
\rm H^i(X/G, \mr H^0(G, F_U)) = \rm H^i_G(X, F_U)$ for each positive $i$. 
 Moreover, $\dim(X/G) \leq \dim(X)$ by \cite[Chapter 6, Proposition 4.11]{pears}. On paracompact Hausdorff spaces, sheaf cohomology coincides with \v{C}ech cohomology, see \cite[Th\'eor\`eme 5.10.1]{godement-faisceaux}. 
 In particular, $\rm H^i(X/G, \mr H^0(G, F_U))  = (0)$ for $i > \dim(X)$. 
\end{proof}
\begin{proof}[Proof of Proposition \ref{prop:finalprop}] Let $p = (p_1, p_2, p_3) \in A(\RR)$. Since $\rm H^1_G(\{p_i\}, \ZZ) = (0)$, we have $\omega|_p = 0$, hence $\omega$ satisfies the topological condition. 
By Theorem \ref{fourierreduction}, 
 it suffices to show that $ \rm H^1(G, \rm H^3(A(\CC), \ZZ(2)))_0 $ is generated by $[\gamma]$ and algebraic elements. We have 
\[
\rm H^1(G, \rm H^3(A(\CC), \ZZ)) \cong \bigoplus_{i + j + k = 3} \rm H^1 \left(G, \rm H^i(E(\CC), \ZZ) \otimes \rm H^j(E(\CC), \ZZ) \otimes \rm H^k(E(\CC), \ZZ) \right). 
\]
Define $\rm H^i_E(j)$ as in (\ref{lastequation}).  
The map $\rm H^2_E(1)^G \otimes \rm H^1(G, \rm H^1_E(1)) \to \rm H^1(G, \rm H^2_E(1) \otimes \rm H_E^1(1))$ is an isomorphism, 
and lifts to $F^0\rm H^2_G(E(\CC), \ZZ(1)) \otimes F^1\rm H^2_G(E(\CC), \ZZ(1)) \to F^1\rm H^4_G(A(\CC), \ZZ(2))$. Consequently, the group $\rm H^1(G, \rm H^i_E \otimes \rm H^j_E \otimes \rm H^k_E) \subset \rm H^1(G, \rm H^3(A(\CC), \ZZ))$ is algebraic for each distinct $i,j,k$ such that $i + j + k = 3$. 
\begin{claim} \label{claim:9.6}
Let $\xi = [\pi_1^\ast(y) \cdot \pi_2^\ast(y) \cdot \pi_3^\ast(y)] \in \rm H^1(G, \rm H^3(A(\CC), \ZZ(2)))$. There exists no lift $\tau \in F^1\rm H^4_G(A(\CC), \ZZ(2))$ of $\xi$ that satisfies topological condition (\ref{topologicalcondition}). 
\end{claim}
Let $\bb S = \set{z \in \CC \mid \va{z} = 1} \subset \CC$. Define $\bb S^+$ to be the differentiable $G$-manifold $\bb S$ with trivial $G$-action, and let $\bb S^-$ be the differentiable $G$-manifold $\bb S$ such that $\sigma(z) = \overline z$. Define $Y =  \bb S^- \times \bb S^- \times \bb S^- $. Any $p \in \bb S^+$ defines a section 
$
s \colon Y \to A(\CC)$ of the projection $ q \colon A(\CC) \to  Y. 
$
For a generator $\eta \in \rm H^3( Y, \ZZ)$, one has $q^\ast([\eta]) = \xi$ with respect to  
$
q^\ast \colon \rm H^1(G, \rm H^3( Y, \ZZ)) \to \rm H^1(G, \rm H^3(A(\CC), \ZZ))$. We have $\dim_{\ZZ/2} \rm H^4_G(Y, \ZZ) = 8$ because the Hochschild-Serre spectral sequence for $Y$ degenerates (since $Y$ is a direct factor of an abelian variety). The map $
\rm H^4_G(Y, \ZZ)  \to \rm H^0(Y^G, \ZZ/2)$ is surjective by Lemma \ref{second-countable} and \cite[\S 1.2.2]{BW20}; because the $\ZZ/2$-ranks on both sides agree, it is an isomorphism. 
If $\tau \in F^1\rm H^4_G(A(\CC), \ZZ)_0$ lifts $\xi = q^\ast([\eta])$, then $s^\ast(\gamma)$ is a lift of $s^\ast q^\ast([\eta]) = [\eta]$ such that $s^\ast(\gamma) |_p = 0 \in \rm H^4_G(\{p\}, \ZZ)$ for $p \in Y^G$. 
This is absurd, and Claim \ref{claim:9.6} follows. Let $\Lambda = \rm H^1(E(\CC), \ZZ)$, 
and consider the basis
\begin{align*}
\set{[x \otimes x \otimes y], [x \otimes y \otimes x], [y \otimes x \otimes x], [y \otimes y \otimes y]} \subset \rm H^1(G, \Lambda \otimes \Lambda \otimes \Lambda). 
\end{align*}
To finish the proof, it suffices to show that 
     $        [x \otimes x \otimes y] - [x \otimes y \otimes x] $ and $ 
              [y \otimes x \otimes x] - [x \otimes y \otimes x] $ are algebraic. Since these two elements are in the same orbit under the action of $\Aut(A)$, it suffices to show that one of them is algebraic. This holds, because we have
       \begin{align*}
x \otimes y - y \otimes x = c_1(\ca P_E)  \in \rm H^1(E(\CC), \ZZ) \otimes \rm H^1(E(\CC), \ZZ) \subset \rm H^2(E^2(\CC), \ZZ),
       \end{align*}
where $\ca P_E$ is the Poincar\'e bundle of $E$ viewed as a line bundle on $E \times E$ via the principal polarization $E \xrightarrow{\sim} E^\vee$, and because $[x] \in \rm H^1(G, \rm H^1(E(\CC), \ZZ(1)))$ is algebraic. 
\end{proof}

\section{Connections with the Griffiths group of an abelian threefold} \label{sec:ceresa}

For a smooth projective variety $X$ over $\CC$, let $\rm{Griff}^2(X) = \CH^2(X)_{\tn{hom}} / \CH^2(X)_{\tn{alg}}$ be the group of homologically trivial codimension-two cycles on $X$ modulo the algebraically trivial ones. 
Non-zero elements in $\rm{Griff}^2(X)$ are difficult to find. For a very general complex abelian threefold $A$, the groups $\rm{Griff}^2(A) \otimes \QQ$ and $\rm{Griff}^2(A) \otimes \ZZ/\ell$ are infinite for every prime number $\ell$ \cite{nori-AV-ceresa, schoen-chownotfinite, rosenschon, totaro-modulo2}. 
These statements rely on non-vanishing of Ceresa cycles on Jacobians of very general curves of genus three \cite{ceresa, hain-ceresa, rosenschon, totaro-modulo2}. In this section, we use different methods to study the subgroup $\rm{Griff}^2(A)^G \subset \rm{Griff}^2(A)$ when $A$ is defined over $\RR$, 
relating this group to the real integral Hodge conjecture. 

For the Jacobian $J(C)$ of a smooth genus three curve $C$ over $\RR$ with $p \in C(\RR)$, one can consider the \emph{Ceresa cycle} $\Xi_p = \iota_p(C) - [-1]^\ast\iota_p(C)$. Here, $\iota_p \colon C \to J(C)$ is the Abel-Jacobi embedding induced by $p$. The resulting class $[\Xi] \in \rm{Griff}^2(J(C)_\CC)^G$ is independent of $p$. 

\begin{proof}[Proof of Theorem \ref{thm:ceresa}]
Let $A$ be a real abelian threefold, and consider the 
second intermediate Jacobian $J^2(A_\CC)$ of $A_\CC$ (see Section \ref{section:AbelJacobi}). Let $J^2(A_\CC)_{\tn{alg}}$ be the algebraic part of $J^2(A_\CC)$ and $J^2(A_\CC)_{tr} = J^2(A_\CC)/J^2(A_\CC)_{\tn{alg}}$ its transcendental part \cite[\S 2.2.2]{voisin}. We get an exact sequence of weight $-1$ Hodge structures $0 \to \Lambda_{\alg} \to \Lambda \to \Lambda_{\tn{tr}} \to 0$; by Theorem \ref{abeljacobicommutativity} and \cite[Corollaire 12.19]{voisin}, 
we obtain a commutative diagram with exact rows:
\begin{align}\label{10:fundamentaldiagram}
\xymatrixrowsep{1.1pc}
\begin{split}
\xymatrix{
 \left(\CH^2(A_\CC)_{\alg}\right)^G \ar[r]\ar[d]& \left(\CH^2(A_\CC)_{\tn{hom}} \right)^G\ar[d] \ar[r]& \rm{Griff}^2(A_\CC)^G\ar[d]  \\
  \rm H^1(G, \Lambda_{\tn{alg}}) \ar[r]& \rm H^1(G, \Lambda) \ar[r] & \rm H^1(G, \Lambda_{\tn{tr}}). 
}
\end{split}
\end{align}
Define $\rm H^1(G, \Lambda)_0 \subset \rm H^1(G, \Lambda)$ as in Definition \ref{isogenydef:subgroups}. The real integral Hodge conjecture for $A$ would imply that $\rm{Griff}^2(A_\CC)^G \otimes \ZZ/2 \to \rm H^1(G, \Lambda_{\tn{tr}})$ surjects onto the image of 
\begin{align}\label{nicemap}
\rm H^1(G, \Lambda)_0 \to \rm H^1(G, \Lambda_{\tn{tr}}).
\end{align}
Our goal is to show that if $A$ is an abelian threefold as in the statement of the theorem, then (\ref{nicemap}) is non-zero, whereas 
if $A = J(C)$ is the Jacobian of a real genus three curve with $p \in C(\RR)$ then the Ceresa cycle $[\Xi_p] \in \CH^2(J(C))_{\tn{hom}}$ maps to zero in $\rm H^4_G(J(C)(\CC), \ZZ(2))$. 

Let us first prove the latter statement. 
Assume that $C$ is hyperelliptic and let $p \in C(\RR)$ be a real Weierstrass point. We have $[\Xi_p] = 0 \in \CH^2(A)$, and for arbitrary $q \in C(\RR)$ there exists $Q \in J(C)(\RR)$ such that the translation by $Q$ map $\tau_Q \colon J(C) \to J(C)$ satisfies $\tau_Q \circ \iota_p = \iota_q$. 
For any abelian threefold $B$ over $\RR$, any cycle $\Gamma \in \CH^2(B)$ and any $Q \in B(\RR)^0$, the cycle $\tau_Q^\ast(\Gamma) - \Gamma \in \CH^2(B)$ is real algebraically equivalent to zero (see \cite[Definition 1.5]{vanhamel}). 
In particular, $[\tau_Q^\ast(\Gamma)] = [\Gamma] \in \rm H^4_G(B(\CC), \ZZ(2))$ by \cite[Lemma 1.2]{vanhamel}. In our case, this gives
\[
0 = 
(\tau_Q)_\ast[\Xi_p] = 
 (\iota_q)_\ast(C) - (\tau_{2Q})_\ast(-1)^\ast(\iota_q)_\ast(C)  = 
[\Xi_q] \in \rm H^4_G(A(\CC), \ZZ(2)). 
\]
Next, suppose that $C$ is a non-hyperelliptic curve of genus three over $\RR$ equipped with $p \in C(\RR)$. 
Let $K \subset \ca M_3(\RR)$ be the connected component containing $[C] \in \ca M_3(\RR)$ (see Section \ref{sec:modtors}). Curves in $K$ have non-empty real locus, and $K$ contains hyperelliptic curves by \cite[Proposition 3.4]{CirreHyperelliptic}. In particular (see \cite{seppalasilhol}, \cite[Proof of Theorem 1.3(B)]{MR4479836}), there exists a contractible complex manifold $B'$, a family of genus three curves $\mr C' \to B'$ with a real structure, and 
two points $a, b \in B'(\RR)$ such that $B'(\RR)$ is connected, $\mr C'_a \cong C$ and $\mr C'_b \eqqcolon C_2$ is hyperelliptic. 

Consider the family of genus three curves $ \pi \colon \mr C' \times_{B'}\mr C' \eqqcolon \mr C \to B \coloneqq \mr C'$. This family $\pi$ has a canonical real structure and equivariant section $\Delta \colon B \to \mr C$. By \cite[Lemma 5.1]{MR4479836}, the real structure on $\pi$ extends to a real structure on the relative Jacobian $J(\mr C)$ of $\mr C$ over $B$; moreover, $\Delta$ induces a $G$-equivariant map $\Psi \colon \mr C \to J(\mr C)$ over $B$. 

Note that $\mr C'(\RR) = B(\RR) \neq \emptyset$ because $C(\RR) \neq \emptyset$. Let $f \colon [0,1] \to B(\RR)$ be a path whose image in $B'(\RR)$ connects $a$ and $b$, and let $q = f(1) \in B(\RR)$. Then $p$ and $q \in B(\RR)$ lie in the same component $K \subset B(\RR)$. The family of Abel-Jacobi maps $\Psi|_{K} \colon \mr C|_{K} \to J(\mr C)|_{K} \to K$ is locally trivial over $K$, so there are $G$-equivariant Lie group diffeomorphisms 
\[
\rho \colon J(C) \cong  J(\mr C_p) \cong J(\mr C_{q}) \cong J(C_2) \; \tn{ with }\;  0 = \rho^{\ast}[\iota_{q}(C_2)] = [\iota_p(C)]  \in \rm H^4_G(J(C)(\CC), \ZZ(2)). 
\]
%
\noindent
We conclude that Ceresa cycles map to zero in equivariant cohomology as desired. 

Let $(A, \lambda)$ be any polarized complex abelian threefold. The embedding of weight $-1$ Hodge structures $\rm H^1(A(\CC), \ZZ(1)) \to \rm H^3(A(\CC), \ZZ(2))$ induced by the polarization $\lambda$ has saturated image, thus gives an embedding of complex tori $\wh A(\CC) = \Pic^0(A)(\CC) \to J^2(A)$. 
\begin{lemma}\label{lemma:charles}
If $(A, \lambda)$ is very general, then $J^2(A)_{\alg}= \wh A$. 
\end{lemma}
\begin{proof}
By \cite[Theorem 17]{peterssteenbrinkmon}, it suffices to show that the representation of $\tn{GSp}_6(\QQ)$ on $\rm H^3(A(\CC), \QQ)_{\tn{prim}}$ is irreducible. This follows from \cite[Theorem 17.5]{fultonharris}. 
\end{proof}
\noindent
Let $\ca A_{3, \delta}$ be the moduli stack of polarized abelian threefolds of polarization type $\delta$, with coarse moduli space $\pi \colon \ca A_{3,\delta} \to \msf A_{3,\delta}$. By Lemma \ref{lemma:charles}, there is a countable union $\cup_iZ_i \subsetneq \msf A_{3, \delta}(\CC)$ of complex analytic subvarieties $Z_i \subset \msf A_{3,\CC}$ with the following property: if the image of a point $[(A, \lambda)] \in \ca A_{3,\delta}(\RR)$ under the map $\pi_\RR \colon \ca A_{3,\delta}(\RR) \to \msf A_{3,\delta}(\RR) \subset \msf A_{3,\delta}(\CC)$ does not lie in $\cup_i Z_i$, then $J^2(A_\CC)_{\tn{alg}} = \wh A(\CC)$. Let $(A, \lambda)$ be such a polarized abelian threefold over $\RR$, and assume $A(\RR)$ is not connected. Via the spectral sequence (\ref{hochschild}), Corollary \ref{corollary:n(A)} and Claim \ref{claim:9.6} imply that the rank of $\rm H^1(G, \rm H^1(A(\CC), \ZZ(1)))$ is smaller than the rank of $\rm H^1(G, \rm H^3(A(\CC), \ZZ(2)))_0$. 
Thus, there are elements in $\rm H^1(G, \Lambda)_0$ which are not in the image of $\rm H^1(G, \Lambda_{\tn{alg}}) \to \rm H^1(G, \Lambda)$, see diagram (\ref{10:fundamentaldiagram}). In particular, the map (\ref{nicemap}) is non-zero, and Theorem \ref{thm:ceresa} follows. 
\end{proof}

\printbibliography

\textsc{Olivier de Gaay Fortman, Institute of Algebraic Geometry, Leibniz University Hannover, Welfengarten 1, 30167 Hannover, Germany}\par\nopagebreak 
  \textit{E-mail address:} \texttt{degaayfortman@math.uni-hannover.de}

\end{document}